\documentclass[11pt]{amsart}
\usepackage{amsmath,amscd,amssymb}
\usepackage[mathscr]{eucal}
\usepackage[all]{xy}

\newtheorem{theorem}{Theorem}[section]
\newtheorem{lemma}[theorem]{Lemma}
\newtheorem{proposition}[theorem]{Proposition}
\newtheorem{corollary}[theorem]{Corollary}

\theoremstyle{definition}

\newtheorem{example}[theorem]{Example}

\theoremstyle{remark}
\newtheorem{remark}[theorem]{Remark}

\numberwithin{equation}{section}

\def\la{\lambda}
\def\al{\alpha}
\def\si{\sigma}
\def\e{\varepsilon}
\def\de{\delta}
\def\NN{{\mathbb N}}
\def\RR{{\mathbb R}}
\def\CC{{\mathbb C}}
\def\TT{{\mathbb T}}

\def\cT{{\mathcal T}}
\def\cS{{\mathcal S}}
\def\cK{{\mathcal K}}
\def\cM{{\mathcal M}}
\def\mD{{\mathcal D}}
\def\Re{{\rm Re}\,}
\def\Im{{\rm Im}\,}
\def\Int{{\rm Int}\,}

\def\conv{{\rm conv}\,}

\def\diag{{\rm diag}\,}
\def\dist{{\rm dist}\,}
\def\usim{\,\smash{\mathop{\sim}\limits^u}\,\,}
\def\mT{\mathcal T}
\def\mS{\mathcal S}
\def\mA{\mathcal A}
\def\mM{\mathcal M}
\def\mK{\mathcal K}

\def\diag{{\rm diag}\,}
\def\DD{{\mathbb D}}

\begin{document}
\title[Diagonals of operators and Blaschke's enigma]{Diagonals of operators and Blaschke's enigma}

\author{Vladimir M\"uller}
\address{Institute of Mathematics,
Czech Academy of Sciences,
ul. \v Zitna 25, Prague,
 Czech Republic}
\email{muller@math.cas.cz}

\author{Yuri Tomilov}
\address{Institute of Mathematics, Polish Academy of Sciences,
\' Sniadeckich str.8, 00-656 Warsaw, Poland}
\email{ytomilov@impan.pl}

\thanks{This work was  partially supported by the NCN grant
2014/13/B/ST1/03153, by the EU grant  ``AOS'', FP7-PEOPLE-2012-IRSES, No 318910, by grant No. 17-27844S of GA CR and RVO:67985840}


\begin{abstract}
We introduce new techniques allowing one to construct diagonals of bounded Hilbert space
operators and operator tuples under ``Blaschke-type'' assumptions.
This provides a new framework for a number of results in the literature
and identifies, often large, subsets in the set of diagonals of \emph{arbitrary} bounded operators (and their tuples).
Moreover, our approach leads to substantial generalizations of the results due to Bourin, Herrero and Stout
having assumptions of a similar nature.
\end{abstract}

\maketitle
\section{Introduction}

Let $T$ be a bounded linear operator on a separable Hilbert space $H.$
If $(e_k)_{k=1}^N$ is an orthonormal basis in $H, 1 \le N\le \infty,$ then $T$ admits a matrix
representation $M_T=(\langle T e_i, e_j\rangle)_{i,j=1}^N.$
If $N<\infty$ then $M_T$ is a finite matrix, and it is  an essential part of the matrix theory to relate the properties of $T$
and $M_T.$ In particular, it is of substantial interest to express
the structure of $T$ in terms of the elements of $M_T.$
Even in this toy setting a number of natural questions, e.g. on the interplay between the spectrum of $T$
and the diagonal of $M_T$ appeared to be rather complicated, see e.g. \cite{Arveson06} and \cite{Tam} for a pertinent discussion.

For an infinite-dimensional space $H$ the relations between $T$ and $M_T$
become even more involved and depend on very advanced methods and techniques
stemming from various domains of analysis. A nice illustration
 of the interplay between $T$ and $M_T$ is provided by the famous Kadison-Singer
problem, its ``matrix'' reformulation and its solution by techniques originating
from matrix theory. We refer e.g. to \cite{Kadison-Singer} for a nice account.

It is well-known that for good enough  (usually Schatten class) $T$
the diagonal of $M_T$ carries a spectral information about $T,$
and for trace-class $T$ its  trace
is just the sum of eigenvalues by Lidskii's theorem.
Note that in this case the trace-class property of $T$ and the value of its trace do not change under a change of basis.
However, the situation becomes more complicated when $T$ is far from being compact.
One way to circumvent the problem is to vary the orthonormal basis $(e_k)_{k=1}^\infty$ and to consider
the whole set $\mathcal D(T):=\{\langle T e_k, e_k\rangle_{k=1}^{\infty}\}\subset \ell_\infty(\mathbb N)$ of diagonals of $T.$

A lot of research has been done
to understand the structure of $\mathcal D(T)$ for general $T$. However, it still remains  mysterious,
even for very particular classes of operators.
The aim of this paper is to introduce and study a Blaschke-type condition  leading to a description of
 ``large'' subsets of $\mathcal D(T)$ in a priori terms. This is done
in a very general set-up of  bounded linear operators $T$ and their tuples.

To put our consideration into a proper framework, let us first review major achievements made so far
in the study of diagonals of operators on infinite-dimensional spaces.
First, note that (as mentioned in \cite{Loreaux16}) the study of diagonals can be understood in at least two senses:

(a) as the study of $\mathcal{D}(T)$ for a class of operators $T,$

and

(b) as the study of  $\mathcal{D}(T)$ for a fixed operator $T.$

Most of the relevant research has been concentrated on the first, easier problem,
but there have been several papers addressing the second problem as well.
 Clearly the
entries constituting diagonals of $T$ belong to the numerical range $W(T)$ of $T$, and thus the numerical range and its subsets appear
to be  very natural candidates for a characterization of at least a part of $\mathcal D(T)$.
However, the importance of numerical ranges for the study of diagonals was underlined only in
\cite{Herrero91}, and to some extent in \cite{Stout81} and \cite{Fan84}.

Motivated by a notorious problem from the theory of commutators,
P. Fan studied in \cite{Fan84} a problem of existence of zero diagonals in $\mathcal D(T).$
Using the properties of the essential numerical range $W_{\rm e}(T)$ of $T,$
he proved that a bounded linear operator $T$ on $H$  admits a
zero-diagonal, that is $(0) \in \mathcal D(T),$ if and only if there exists an orthonormal basis $(e_k)_{k=1}^\infty$
such that  $s_n:=\sum_{k=1}^{n} \langle T e_k, e_k \rangle, n \in \mathbb N,$
possess a subsequence $(s_{n_m})_{m=1}^{\infty}$ satisfying $s_{n_m} \to 0, m \to \infty.$
A number of related results in \cite{Fan84} suggested that
``the diagonal of an operator carries more information about the operator than its
relatively small size (compare to the "fat" matrix representation of the operator)
may suggest.'' This line of research was continued e.g. in  \cite{Fan87} and \cite{Fan94}.
By different methods, it has been shown recently in \cite{Loreaux16} that an infinite-rank idempotent
admits a zero-diagonal if and only if it is not  a Hilbert-Schmidt perturbation of a (selfadjoint) projection.

An systematic study of the set $\mathcal D(T)$ has been attempted  by Herrero in \cite{Herrero91}.
Addressing the challenging problem (b) above, Herrero showed that if $\{d_n\}^\infty_{n=1}$ belongs to the interior $\Int W_{\rm e}(T)$ of $W_{\rm e}(T)$
and $\{d_n\}^\infty_{n=1}$ has a limit point again in $\Int W_{\rm e}(T)$, then $\{d_n\}^\infty_{n=1} \in \mathcal D(T)$.
Otherwise, as shown in \cite{Herrero91}, there exists a
compact operator $K$ such that $\{d_n\}^\infty_{n=1} \in \mathcal D(T+K)$.  Finally, if ${\rm dist}\{d_n,W_{\rm e}(T)\}\to 0$ as $n\to\infty$
then there exists $\{d'_n\}^\infty_{n=1}$  such that $\{d'_n\}^\infty_{n=1} \in \mathcal D(T)$ and
$|d_n-d'_n|\to 0$ as $n\to\infty$. The results by Herrero motivated the research in this paper to a large extent.
In particular, the importance of Blaschke-type conditions was suggested by  analysis of diagonals for the unilateral shift in \cite{Herrero91}. See Section \ref{diagon} for more on that.
As remarked in \cite{Herrero91}, the numerical range results allow one to deduce easily Fong's theorem from \cite{Fong86}
saying that every bounded sequence  $(d_n)_{n=1}^\infty$ admits a nilpotent operator (of index $2$)  whose diagonal
is  $(d_n)_{n=1}^\infty$. Moreover, in the same way, one can derive a similar result from \cite{Loreaux16} replacing nilpotents by idempotents.
See Remark \ref{diagon_rem} for more on that.

The papers by Fan and Herrero were preceded by a deep article by Stout \cite{Stout81}, where the diagonals of $T$ appeared in a natural way in the study of Schur algebras
and where $\mathcal D(T)$ was also related to  $W_{\rm e}(T).$ Recall that  $0 \in W_{\rm e}(T)$ if (and only if) there exists $(d_n)_{n=1}^{\infty} \in \mathcal D(T)$ such that $(d_n)_{n=1}^{\infty} \in c_0(\mathbb N).$
At the beginning of $1970$s, Anderson proved in \cite{AndersonThesis} that $0 \in W_{\rm e}(T)$ is in fact equivalent to the existence of a $p$-summable sequence in $\mathcal D(T)$ for every $p>1.$ Extending Anderson's result, Stout discovered  that $0 \in W_e(T)$ yields
a stronger property:  for any $(\alpha_n)_{n=1}^{\infty}
\not \in \ell_1$ there is $(d_n)_{n=1}^{\infty} \in \mathcal D(T)$ such that $|d_n|\le \alpha_n$ for all $n.$
These results were crucial in the study of matrix and Schur norms for Hilbert space operators in \cite{Fong87}.
 In particular, by relating the size of entries of $M_T$ to $W_{\rm e}(T),$  the infima of maximum entry
norm and Schur norms of $M_T$ over all
choices of bases in $H$ were proved in \cite{Fong87} to be precisely ${\rm dist}\, \{0, \sigma_{\rm ess}(T)\},$
where $\sigma_{\rm ess}(T)$ stands for the essential spectrum of $T.$

 Comparatively recently, a relevant study of $\mathcal D(T)$ in a very general context of \emph{operator-valued} diagonals
  has been realized in \cite{Bourin03} by J.-C. Bourin. In particular, he proved in \cite[Theorem 2.1]{Bourin03} that if $W_{\rm e}(T)$  contains the open unit disc $\DD$, then for every sequence $\{C_n\}_{n\ge1}$ of  operators
  acting possibly on different Hilbert spaces and satisfying
    $\sup\limits_{n\ge1}\Vert C_n\Vert<1$, there exists a sequence of mutually orthogonal subspaces $(M_n)_{n\ge 1}$ of $H$ such that $\bigoplus_{n\ge 1}M_n=H$ and the compressions $P_{M_n} T\vert_{M_n}$ of $T$ are unitarily equivalent to $C_n$ for all $n\ge1.$ The operators $\bigoplus_{n \ge 1} \bigl(P_{M_n} T\vert_{M_n}\bigr),$ called ``pinchings'' of $T$ in \cite{Bourin03}, can be considered as  operator counterparts of elements from $\mathcal D(T),$ and they essentially coincide with those  elements when $C_n$ act on one-dimensional spaces. Note that pinching results appeared to be useful in the
  study of $C^*$-algebras, see e.g. \cite{Bourin16} for more on that. For recent generalizations of these results see \cite{Muller19}.

A rather general but unfortunately only approximate result on diagonals has been obtained by
Neumann in \cite{Neumann99} where, in particular,
 the $\ell^\infty$-closure of $\mathcal D(T)$ for (bounded) selfadjoint $T$
was identified with a certain convex set. More precisely,
let $T$ be a bounded selfadjoint operator on $H,$ and let
$t_-=\inf \sigma_{\rm ess}(T)$ and $t_+=\sup \sigma_{\rm ess}(T).$ Then
$T=T_- + T_0 + T_+,$ where $T_+$ and $T_-$ are compact selfadjoint operators with positive and negative spectra
respectively. Denoting ${\bf t}_+\  ({\bf t}_-)$ the diagonal of $T_+\  (T_-)$
with respect to some orthonormal basis, it was shown in \cite{Neumann99}  that
the closure of  $\mathcal D(T)$ in $\ell^{\infty}(\mathbb N)$ can be described as
$$
\overline{{\rm conv}\, S {\bf t}_-} +[t_-,t_+]^{\mathbb N} + \overline{{\rm conv}\, S {\bf t}_+},
$$
where $S$ is a group of permutations (bijections) of $\mathbb N.$
As shown in \cite{Neumann99},  this result yields, in particular, a description of closed convex invariant subsets for selfadjoint operators.
The results of a similar and more general nature have appeared recently in \cite{Kennedy15}.

A different perspective  has been opened since Kadison's striking results on diagonals of selfadjoint projections, addressing the problem (a).
Generalizing the classical  Schur-Horn theorem for matrices, Kadison proved  in \cite{Kadison02a, Kadison02b}
that a sequence $(d_n)_{n=1}^\infty$ is
a diagonal of some \emph{selfadjoint} projection if and only if it takes values in $[0, 1]$ and if
the sums $a := \sum_{d_j < 1/2} d_j$ and $b:=\sum_{d_j \ge 1/2} (1-d_j)$ satisfy either $a+b=\infty$ or $a+b <\infty$ and $a-b \in \mathbb Z.$
Illustrating the sharpness of Kadison's result and drawbacks of Neumann's approximate description, we note
that the Neumann's theorem would produce all positive sequences with elements between $0$ and $1.$
The Kadison integer condition was recognized as the Fredholm index obstruction in a subsequent paper by Arveson \cite{Arveson07},
where Kadison's dichotomy was studied for normal operators with finite spectrum.
 (Note also \cite{Kaftal17}, where the integer was identified with so-called essential codimension of a pair of projections.)
More precisely,
Arveson considered a more general task of describing the diagonals
for normal operators  $N(X)$ with the essential spectrum $\sigma_{\rm e}(T) =\sigma(T)=X$ and with $X$ being the set of vertices of a convex polygon $P\subset\mathbb C.$ He defined  the set  ${\rm Lim}^1(X)$ of ``critical'' sequences $d=(d_n)_n\in l^\infty(\Bbb N)$ whose limit points belong to $X$ and moreover such that the elements of $d$ converge rapidly to these limits in the sense that $\sum_{n=1}^\infty {\rm dist}(d_n,X)<\infty$,
that is $(d_n)_{n=1}^\infty$ satisfies the analogue of the classical Blaschke condition. He showed that there is a discrete group $\Gamma_X$
depending
only on the arithmetic properties of $X,$ and a surjective mapping
$d \mapsto s(d)\in \Gamma_X$ of the set of all such sequences $d$ such that if $s(d) \neq 0,$ then $d$ is not the diagonal of any
operator in $N(X).$ Moreover, it appeared that  this is the only
obstruction in the case of two-point sets, but also that there are other  obstructions in
the case of three-point sets. The case of three and more points set was settled very recently in
\cite{Loreaux19} and we refer to this paper for very interesting details.
We also refer to \cite{Argerami15} containing an illuminating discussion of Kadison's theorem.

The research started by Kadison and Arveson gave rise to an intensive activity around diagonals of operators.
In particular,
the set $\mathcal D(T)$ was characterized for several  classes of $T$ (bearing certain resemblance
to the selfadjoint situation)
and moreover several results (e.g. Neumann's $l^\infty$-closure result discussed above) were extended to the setting of von Neumann algebras and tuples of elements, sometimes obtaining new statements in the case of operators as well, see e.g.
\cite{Argerami13}, \cite{Kennedy15}, \cite{Massey16}, and \cite{Kennedy17} and the references therein.
Without mentioning all the substantial contributions,
we give only a few samples related to the above discussion.
Pairs of null real sequences realized as sequences of eigenvalues and diagonals of positive compact operators were characterized in \cite{Kaftal10}.
A description of diagonals for selfadjoint operators $T$ with finite spectrum was recently given in \cite{Bownik15} (see also \cite{Bhat14}).
Note that in this case the description of $\mathcal D (T)$ though explicit becomes rather technical and has an involved
formulation.
Very recently, $\mathcal D(T)$  for a class of unitary operators $T$ was described in  \cite{Loreaux18}.
In particular, it was shown in \cite{Loreaux18} that
a complex-valued sequence $(d_n)_{n=1}^\infty$
is a diagonal of some unitary operator on $H$ if and only if $\sup_{n \ge 1} |d_n|\le 1 $
and
\begin{equation}\label{diag_unit}
2(1-\inf_{n \ge 1} |d_n|) \le \sum_{n=1}^{\infty}(1-|d_n|).
\end{equation}
One should also mention the applications to frame theory where the diagonals arise e.g. as sequences of
frame norms. For this direction of research one may consult e.g. \cite{Antezana07}, \cite{Bownik11} and \cite{Bownik15a}.
For several other related results, we refer to a recent survey \cite{Weiss14}.

In this paper we study the diagonals of bounded operators from the perspective of numerical ranges and spectrum.
Being inspired by Herrero's work \cite{Herrero91} and by Arveson's considerations in \cite{Arveson07},
we assume that the diagonal belongs to the interior of $W_{\rm e}(T)$
and we introduce
\emph{a Blaschke-type} condition
$$
\sum_{n=1}^{\infty} {\rm dist} \, \{d_n, \partial W_{\rm e} (T)\}=\infty
$$
on the size of diagonal $(d_n)_{n=1}^\infty$ near the boundary of
$W_{\rm e}(T).$  (The terminology originates from an \emph{opposite} condition from complex analysis
concerning zeros of bounded analytic functions in the unit disc. Rather then writing ``non-Blaschke'' here and in the sequel
we have decided to name the condition above ``Blaschke-type''.)
Note that one can easily recognize ``Blaschke's'' component in \eqref{diag_unit}, however the condition is more subtle.

Given $T \in B(H),$ this set-up helps us to suggest a general and new method for constructing
a big part of diagonals that works in a variety of different settings, including operator tuples  and operator-valued diagonals.
Moreover, for commuting tuples it allowed us to use spectral properties of $T$ for constructing diagonals
for power tuples.
Thus, apart from substantially generalizing the results by Herrero and Bourin,
we propose an approach that unifies and extends existing results  and does not depend on specific properties of $T$
(e.g. as being selfadjoint or unitary). We would like to stress that, dealing with problem (b) above,
 we work with fixed operators rather than operator classes thus
dealing with a more demanding and involved task.
On the other hand, the drawback of our framework is that
working with a fixed operator we are still far from giving a full description of $\mathcal D(T)$ for a fixed $T$,
 and results similar to Kadison's ones are out of reach.
 Although we suspect a characterization of $\mathcal D(T)$ is hardly possible in such a generality, see also Section
 \ref{diagon} for a discussion of necessary conditions for diagonals and some classes of $T$
 when Blaschke-type assumptions help to obtain  exhaustive characterizations.

Our first two main results given below provide Blaschke-type conditions
for a sequence from the essential numerical range  of a tuple $\cT$ to be a diagonal of $\cT$.
To treat degenerate cases, we deal with the notion of  interior $\Int_M$ relative to an affine real subspace $M$ in $\mathbb C^n$  first, and the result for interior in $\mathbb C^n$ itself
arises as a corollary.

\begin{theorem}\label{T1.1}
Let $\cS=(S_1,\dots,S_s)\in B(H)^s$ be an $s$-tuple of selfadjoint operators, and let $(\la_k)_{k=1}^\infty\subset\Int_{\RR^s} W_{\rm e}(\cS)$ satisfy
\begin{equation}\label{blaschke}
\sum_{k=1}^\infty\dist\{\la_k,\RR^s\setminus W_{\rm e}(\cS)\}=\infty.
\end{equation}
Then $(\la_k)_{k=1}^\infty \in\mathcal D(\cS)$.
\end{theorem}

Theorem \ref{T1.1} has a counterpart for $\cT=(T_1,\dots,T_n)\in B(H)^n.$
Note that in a natural way $W(\cT)$ can be identified with a subset of $\mathbb R^{2n}.$
If $\cT \in B(H)^n$ then,
assuming that $M\subset\mathbb R^{2n}$ is the smallest affine subspace  containing $W(\cT)$, we conclude
in Corollary \ref{complex} below that any $(\la_k)_{k=1}^{\infty} \subset\Int_M W_{\rm e}(\cT)$ satisfying
$$
\sum_{k=1}^\infty \dist\{\la_k,M\setminus W_e(\cT)\}=\infty,
$$
belongs to $\mathcal D(\cT)$.

Using Theorem \ref{T1.1} in the context of power tuples $(T, \dots, T^n)$ for $T\in B(H)$, and thus being able to invoke
the notion of spectrum, we  show that
for every $(\la_k)_{k=1}^\infty\subset\Int\widehat\si(T)$ satisfying  $\sum_{k=1}^\infty \dist^n\{\la_k,\partial\widehat\si(T)\}=\infty$ the sequence $(\la_k,\la_k^2,\dots,\la_k^n)$
 belongs to $\mathcal D (T, \dots, T^n).$ We are not aware of similar statements in the literature, although some related statements
 can be found in \cite{Muller19}.

\begin{theorem}\label{T2.1}
Let $\cT=(T_1,\dots,T_n)\in B(H)^n$. For every $(\la_k)_{k=1}^{\infty}\subset W_{\rm e}(\cT)$
and every $(\al_k)_{k=1}^{\infty}\notin\ell_1$  there exists an orthonormal basis $(u_k)_{k=1}^{\infty}$ in $H$ such that
\begin{equation}\label{stout}
\bigl\|\langle \cT u_k,u_k\rangle-\la_k\bigr\|\le |\al_k|
\end{equation}
for all $k\in\NN$.
\end{theorem}

Theorem \ref{T2.1} yields the existence of diagonals for compact perturbations of tuples
that satisfy weakened Blaschke-type conditions.
In particular, if $p>1$ and $(\la_k)_{k=1}^{\infty}\subset \CC^n$ satisfy
$\sum_{k=1}^\infty \dist^p\{\la_k,W_{\rm e}(\cT)\}<\infty$, then
there exists an $n$-tuple of operators $\cK=(K_1,\dots,K_n)$ with $K_j, 1 \le j \le n,$ from the Schatten class $S_p$ such that
$(\la_k)_{k=1}^{\infty} \in \mathcal D (\cT+\cK)$. Our results on compact perturbations
lead to several characterizations of the subset $\mathcal D_{\rm const}(T)$ of $\mathcal D(T)$ consisting of constant diagonals.
The problem of understanding the structure of $\mathcal D_{\rm const}(T)$ has been raised in \cite{Bourin03},
although there are other related results in the literature. The most interesting question on convexity of
$\mathcal D_{\rm const}(T)$ remains still unanswered.

It is curious that the same type of technique yields the results in the context of operator diagonals
thus extending Bourin's results from \cite{Bourin03} mentioned above to the setting of tuples
and replacing his uniform contractivity condition on operator diagonal by a more general assumption of Blaschke's type.
In particular, the following statement holds.

\begin{theorem}\label{T3.1}
Let $T\in B(H)$ with $W_{\rm e}(T)\supset \overline {\DD}$. Let $L_k, k\in\NN,$ be separable Hilbert spaces (finite or infinite-dimensional) and  let $C_k\in B(L_k)$ be contractions satisfying $\sum_{k=1}^\infty (1-\|C_k\|)=\infty$. Then there exist projections $P_{K_k}, k \in \mathbb N,$ onto mutually orthogonal subspaces $K_k\subset H$ such that $\bigoplus_{k=1}^\infty K_k=H$ and
the compression $P_{K_k}T\vert_{K_k}$
is unitarily equivalent to $C_k$
for all $k\in\NN$.
\end{theorem}

(Note that since $W_{\rm e}(T)$ is compact, in the formulation above, one may replace the closure of ${\DD}$ by $\DD$ itself ,
and the choice of $\overline {\DD}$ is just a matter of taste.)

The methods used in the proof of Theorem \ref{T3.1} appeared to be quite fruitful in the study of power tuples.
Under natural spectral assumptions, they allowed us also to construct operator diagonals of power tuples $(T,\dots,T^n)$
consisting for power tuples of contractions $(C_k,\dots,C_k^n)$ where $(C_k)_{k=1}^{\infty}$ satisfy an analogue for tuples of Blaschke's condition in Theorem \ref{T3.1}.
 More precisely,
assuming that $\widehat\si(T)\supset\overline{\DD}$, $n\in\NN$, and  $L_k, k \in \mathbb N,$ are  separable Hilbert spaces, we prove that  for any contractions
$C_k\in B(L_k), k \in \mathbb N,$ satisfying $\sum_{k=1}^\infty (1-\|C_k\|)^n=\infty$
one can find mutually orthogonal subspaces $K_k\subset H$ such that
$
H=\bigoplus_{k=1}^\infty K_k
$
and $P_{K_k}(T, \dots, T^n)\vert_{K_k}$ is unitarily equivalent to $(C_k,\dots,C_k^n)$
for all $k\in\NN$.

The result has a flavor of (finite) power dilations. Note however that its proof relies on the classical unitary power dilation for a Hilbert space
contraction.

\section{Notation}\label{nota}
It will be convenient to fix some notations in a separate section. In particular, we let
$H$ be an infinite-dimensional complex separable Hilbert space with the inner product $\langle\cdot ,\cdot \rangle,$ and $B(H)$ the space of all bounded linear operators on $H$.
For a bounded linear operator $T$ we denote by $\sigma (T)$ its spectrum, by $W(T)$ its numerical range, and by $N(T)$ its kernel.
For a set $M$ and $H_m\subset H, m\in M,$ denote by $\bigvee_{m \in M} H_m$ the smallest (closed) subspace containing all of $H_m.$

In the following we consider an $n$-tuple $\mathcal T=(T_1,\dots,T_n)\in B(H)^n$. Note that we do not in general
assume that the operators $T_j$ commute. For $x,y\in H$ we write shortly
$$\langle \mathcal Tx,y\rangle= (\langle T_1x,y\rangle,\dots,\langle T_nx,y\rangle)\in\CC^n  \quad \text{and} \quad {\mathcal T}x=(T_1x,\dots,T_nx)\in H^n.$$
Similarly for $\la=(\la_1,\dots,\la_n)\in\CC^n$
we write $\mathcal T-\la=(T_1-\la_1,\dots,T-\la_n)$ and
\begin{equation}\label{lam}
\|\la\|=\max\{|\la_1|,\dots,|\la_n|\}.
\end{equation}
If  $\mT=(T_1,\dots, T_n)\in B(H)^n$ and $R, S \in B(H)$ then we define
\begin{equation}\label{multip}
 R \mT  S :=(R T_1 S, \dots, R T_n S).
\end{equation}
 For a subspace $M$ of a Hilbert space $H$  we denote by $P_M$ the orthogonal projection onto $M.$

If $\mT=(T_1,\dots, T_n)\in B(H)^n$ and $\mS=(S_1,\dots, S_n)\in B(H')^n$ are $n$-tuples of operators,
then we say that $\mT$ and $\mS$ are unitarily equivalent and write $\mT\usim\mS$ if there exists a unitary operator $U:H\to H'$ such that $T_j=US_jU^{-1}$ for every $j=1,\dots,n.$

For $\mT=(T_1,\dots, T_n)\in B(H)^n$ we denote by $\mathcal D (\mathcal T)\subset \ell_\infty(\NN, \mathbb C^n)$ its set of diagonals. In other words,
$$
\mathcal D (\mathcal T)=\left \{(\langle T_1 e_n, e_n\rangle, \dots,\langle T e_n, e_n \rangle)_{n=1}^\infty \right\}
$$
when  $(e_n)_{n=1}^\infty$ varies through the set of  all orthonormal bases of $H.$
It will be important to distinguish a subset of $\mathcal D (\mathcal T)$ consisting of "constant" diagonals. Define
$$
\mD_{\rm const} (\cT)=\{
\la\in\CC^n: (\la,\la,\dots)\in \mD (\cT)\}.
$$

For a closed set $K \subset \mathbb C^n $ we denote by $\partial K$ the topological boundary of $K$,
by ${\rm conv}\, K$ the convex hull of $K,$ and by $\widehat K$
the polynomial hull of $K$. Recall that if $K\subset\CC$ is compact, then $\widehat K$ is the union of $K$ with all bounded components of the complement $\CC\setminus K$.
For $K\subset\CC^n$ denote by ${\rm Int} \, K$ the interior of $K$. If $K\subset M\subset\CC^n$ then denote
by ${\rm Int}_M \, K$ the relative interior with respect to $M \subset \mathbb C^n$ with the induced topology.
By an affine subspace $M\subset \mathbb R^n$ we mean, as usual, any set of the form $a+M_0,$
where $a$ is a fixed vector from $\mathbb R^n,$ and $M_0$ is a subspace of $\mathbb R^n.$

Finally, we let $\mathbb T$ stand for the unit circle $\{\lambda \in \mathbb C: |\lambda|=1\}$, $\mathbb D$ for the unit disc $\{\lambda\in\CC: |\lambda|<1\}$ and $\mathbb R_+$ for $[0,\infty).$ For $\rho>0,$ write $\TT_\rho=\{z\in\CC: |z|=\rho\}$.

\section{Preliminaries}\label{prelim}
We start with recalling certain basic notions and facts from the spectral theory of operator tuples on Hilbert spaces.
They can be found e.g. in \cite[Chapters 2-3]{Muller07}.

Let $\mathcal T=(T_1,\dots,T_n)\in B(H)^n$ be an $n$-tuple of commuting operators.
Recall that  its  joint (Harte) spectrum $\sigma(\mathcal T)$ can be defined as the complement of the set of those $\lambda=(\la_1,\dots,\la_n) \in\mathbb C^n$ for which
$$
\sum_{j=1}^n L_j(T_j-\lambda_j)=\sum_{j=1}^{n}(T_j-\lambda_j)R_j=I
$$
for some $L_j, R_j, 1 \le j \le n,$ from the algebra $B(H).$
It is well-known that $\sigma({\mathcal T})$ is a non-empty compact subset of $\mathbb C^n.$
One can define the joint essential spectrum $\si_e(\mathcal T)$ as the (Harte) spectrum of the $n$-tuple $(T_1+{\mathcal K}(H),\dots,T_n+{\mathcal K}(H))$ in the Calkin algebra $B(H)/{\mathcal K}(H)$, where ${\mathcal K}(H)$ denotes the ideal of all compact operators on $H$.
This definition of $\sigma_{\rm e}(\mathcal T)$ is rather implicit. Thus
it is  helpful to consider the joint essential approximate point spectrum
$\sigma_{\pi {\rm e}}(\mathcal T) \subset \si_{{\rm e}}(\mathcal T)$ defined as the set of all $\la=(\la_1,\dots,\la_n)\in\CC^n$ such that
$$\inf_{x\in M,\|x\|=1}\sum_{j=1}^n\|(T_j-\la_j)x\|=0$$
for every subspace $M\subset H$ of finite codimension.
 The set $\si_{\pi {\rm e}}({\mathcal T})$ is quite a big part of $\si_{{\rm e}}({\mathcal T})$ so that the polynomial convex hulls $\widehat \sigma_{\rm e}({\mathcal T})$ and $\widehat \sigma_{\pi {\rm e} }(\mathcal T)$ coincide.
Note that if $n=1,$ then
$\si_{\rm e}(T_1)=\{\la_1\in\CC: T_1-\la_1\hbox{ is not Fredholm}\}$
and
$\sigma_{\pi,e}(T_1)=\{\la_1\in \CC: T_1-\la_1 \hbox{ is not
upper-semi-Fredholm}\}.$
For $T\in B(H)$ and $\mathcal T=(T,T^2, \dots, T^n) \in B(H)^n,$ one has $\sigma (\mathcal T)=\{(\lambda, \dots, \lambda^n): \lambda \in \sigma (T)\}$ and
$\sigma_{\rm e} (\mathcal T)=\{(\lambda, \dots, \lambda^n): \lambda \in \sigma_e (T)\}$, and the same property holds for $\sigma_{\pi {\rm e}}(T).$
Denote by $\si_p(\mathcal T)$ the point spectrum of $\mathcal T$, i.e., the set of all $n$-tuples $\lambda=(\la_1,\dots,\la_n)\in\CC^n$ such that $\bigcap_{j=1}^n N(T_j-\la_j)\ne\{0\}$. If $ x \in \bigcap_{j=1}^n N(T_j-\la_j)$ then
we will write $\mT x=\lambda x.$

It is often useful to relate $\sigma(\mathcal T)$ to a larger and more easily accessible set $W(\mathcal T)\subset \mathbb C^n$ called the  joint numerical
range of $\mathcal T$ and defined as
$$
W(\mathcal T)=\bigl\{(\langle T_1 x, x\rangle , ..., \langle T_n x, x \rangle) : x \in H, \|x\|=1\bigr\}.
$$
The set $W(\mathcal T)$ can be identified with a subset of $\mathbb R^{2n}$ if one identifies the $n$-tuple $\mathcal T$ with the $(2n)$-tuple
$({\rm Re}\, T_1, {\rm Im}\, T_1, ..., {\rm Re}\, T_n, {\rm Im}\, T_n)$ of selfadjoint operators.
Unfortunately, if  $n>1,$ then $W(\mathcal T)$ is not in general convex, see e.g. \cite{Li09}.

As in the spectral theory, there is also a notion of the joint essential numerical range  $W_{\rm e}(\mathcal T)$  associated to $\mathcal T.$
For $\mathcal T=(T_1,\dots,T_n)\in B(H)^n$ we define  $W_{\rm e}(\mathcal T)$ as the set of all $n$-tuples
$\la=(\la_1,\dots,\la_n)\in\CC^n$ such that there exists an orthonormal sequence $(x_k)_{k=1}^{\infty}\subset H$ with
$$
\lim_{k\to\infty}\langle  T_j x_k,  x_k\rangle=\la_j, \qquad j=1,\dots,n.
$$
It is instructive to note that in the definition above one may choose  $(x_k)_{k=1}^{\infty}\subset H$ to be an orthonormal basis of $H.$
For $n=1$ the proof of the latter fact can be found in \cite{Bourin03} and \cite{Stout81},
for the general case see \cite[Theorem 2.1]{Li09}.
It is easy to see that $\la \in W_{\rm e}(\cT)$ if and only if
for every $\de>0$ and every  subspace $M\subset H$ of finite codimension
there exists a unit vector $x\in M$ such that $\|\langle\mT x,x\rangle-\la\|<\de,$
see e.g. \cite[Proposition 5.5]{Muller19} for the proof.
The latter property was a basis of many inductive constructions in \cite{Muller18} and \cite{Muller19},\
and it will also be crucial in this paper.

Alternatively, $W_{\rm e}(\mathcal T)$ can be described as
$$
W_{\rm e}(\mathcal T):= \bigcap \overline{W(T_1 + K_1,\dots,T_n+K_n)}
$$
where the intersection is taken over all $n$-tuples $K_1,\dots,K_n$ of compact operators on $H.$
Moreover, by \cite[Corollary 14]{MullerSt}, we can always find a tuple $(K^0_1,\dots,K^0_n) \in K(H)^n$ such that
\begin{equation}\label{compactp}
W_{\rm e}(\mathcal T)= \overline{W(T_1 + K^0_1,\dots,T_n+K^0_n)}.
\end{equation}
Recall that $W_{\rm e}(\mathcal T)$ is a nonempty, compact and, in contrast to $W(\mathcal T),$ \emph{convex} subset of $\overline{W(\mathcal T)}$, see \cite{BercoviciFoias} or \cite{Li09}.

As a straightforward consequence of the definitions above, if the $n$-tuple $\mathcal T \in B(H)^n$ is commuting then  $\si_{\pi e}(\mathcal T)\subset W_{\rm e}(\mathcal T).$ Since the polynomial hulls of $\si_{\pi e}(\mathcal T)$ and $\si_{e}(\mathcal T)$ coincide \cite[Corollary 19.16]{Muller07},
their convex hulls coincide as well, and the convexity of $W_{\rm e}(\cT)$ yields
\begin{equation}\label{sigmae}
{\rm conv} \, \sigma_{\rm e}(\cT)\subset W_e(\cT).
\end{equation}
For a comprehensive account of essential numerical ranges one may consult \cite{Li09}. See also \cite{Muller19} for
a discussion of other numerical ranges closely related to the notion of essential numerical range.

There are several other numerical ranges useful in applications. In particular, the so-called infinite numerical range $W_\infty(\mathcal T)$
will be relevant in the sequel.
 Recall that if $\mathcal T=(T_1, \dots, T_n) \in B(H)^n$ then $W_\infty(\mathcal T)$  can be defined by
 $$
W_\infty(\mathcal T):=\bigl\{(\lambda_1, \dots, \lambda_n)\in \mathbb C^n: P T_jP =\la_jP, \quad j=1,\dots,n\bigr\}
$$
for some infinite rank projection $P.$
Note that $\la \in W_{\infty}(\cT)$ if and only if
for  every  subspace $M\subset H$ of finite codimension
there exists a unit vector $x\in M$ such that $\langle\mT x,x\rangle=\la,$
see e.g. \cite[Proposition 5.4]{Muller19}.
By \cite[Theorem 5.8]{Muller19}, the essential numerical range $W_{\rm e}(\mT)$ of $\mathcal T$ can be described
in terms of $W_\infty (\mathcal T)$ as
$$
W_e(\mT)=\bigcup_{\mK\in\mK(H)^n} W_\infty(\mT+\mK).
$$

The infinite numerical range of a tuple is closely related to its essential numerical range as the following statement shows,
see \cite[Corollary 4.3]{Muller18} and \cite[Corolary 4.3]{Muller19}.

\begin{theorem}\label{wint}
For any $\mathcal T \in B(H)^n,$
\begin{equation}\label{einf}
\Int (W_{\rm e}(\mathcal T)) \subset W_{\infty}(\mathcal T).
\end{equation}
Moreover, if the tuple $\mathcal T$ is commuting then
\begin{equation}\label{a2}
\Int \conv\, \si(\mathcal T) \subset W(\mathcal T).
\end{equation}
\end{theorem}

Thus $W_\infty(\mT)$ is large if $W_{\rm e}(\mathcal T)$ is large.
Note in passing that in general, by \cite[Theorem 4.2]{Muller19},
\begin{equation}\label{a1}
\Int  \conv\, \bigl(W_{\rm e}(\mathcal T)\cup\si_p(\mathcal T)\bigr)\subset W(\mathcal T).
\end{equation}

The importance of $W_\infty(\mathcal T)$ can be illustrated by the next result crucial for our subsequent arguments,
 For $S\subset\CC^n$  denote by $\mM(S)$ the set of all $n$-tuples of operators $\mA=(A_1,\dots,A_n)\in B(H)^n$  such that
 there exist an orthonormal basis $(x_k)_{k=1}^{\infty}$ in $H$ and a sequence $(\la_k)_{k=1}^{\infty}\subset S$
 satisfying  $\mA x_k=\la_k x_k$ for each $k$.
The theorem below identifies compressions of a tuple $\mT$ with a tuple of diagonal operators $\mA$
whose diagonals belong to the infinite numerical range of $\mT,$ its proof can be found in \cite[Proposition 6.2]{Muller19}.

\begin{proposition}\label{propm}
Let $\mT=(T_1,\dots,T_n)\in B(H)^n$. Let $\mA\in\conv\mM(W_\infty(\mT))$. Then there exists a subspace $L\subset H$ such that the compression
 $P_L \mT\vert_{L}$ is unitarily equivalent to $\mA$.
\end{proposition}

Tuples $\mathcal T_n=(T, T^2,\dots,T^n)$ consisting of powers of a single operator $T \in B(H)$ are of special interest
since they allow one to reveal the structure of an operator $T$ by looking at its powers, thus sometimes uncovering new effects
(see e.g. \cite{Muller18}).
The following statement from \cite{Muller18} describes big subsets of $W(T,T^2,\dots,T^n)$ in terms
of the \emph{spectrum} of $T$, see \cite[Theorem 4.6]{Muller18}. In this paper, it will be vital for constructing (operator) diagonals for tuples of operator powers.

\begin{theorem}\label{lambdawe}
Let $T\in B(H)$ and let $\la$ belong to the interior of the polynomial hull $\widehat \sigma (T)$ of $\sigma(T)$.  Then
$$
(\la,\la^2,\dots,\la^n)\in  \Int W_{\rm e}(T, T^2, \dots, T^n) \subset W_\infty(T,T^2,\dots,T^n).
$$
for all $n\in\NN.$
\end{theorem}

\section{Diagonals: Blaschke-type condition}\label{diagon}

Let $T\in B(H)$. Recall from the introduction that according to  \cite{Herrero91}, any sequence $(\la_k)_{k=1}^\infty\subset \Int W_{\rm e}(T)$ with an accumulation point inside $\Int W_{\rm e}(T)$ can be realized as a diagonal of $T$, i.e., there exists an orthonormal basis $(u_k)_{k=1}^{\infty}$ in $H$ such that $\langle Tu_k,u_k\rangle=\la_k$ for all $k$.

Below we prove  a similar result under a much weaker, Blaschke-type assumption:
$$
\sum_{k=1}^\infty\dist\{\la_k,\partial W_{\rm e}(T)\}=\infty.
$$
Note that if  $W_{\rm e}(T)=\overline\DD$ then this assumption reduces to the condition $\sum_{k=1}^\infty(1-|\la_k|)=\infty$,
opposite to the classical Blaschke's one.
Moreover, our technique allows us to obtain the  result in a more demanding setting of operator tuples, i.e., we construct given common diagonals for $n$-tuples of operators $\mathcal T$ with respect to a common orthonormal basis.

As mentioned in the previous section, the (essential) numerical range of an $n$-tuple $(T_1,\dots,T_n)\in B(H)^n$ can be identified with the (essential) numerical range of the $(2n)$-tuple $(\Re T_1, \Im T_1,\dots, \Re T_n, \Im T_n)$ of selfadjoint operators, considered as a subset of $\RR^{2n}$. It will be convenient to formulate the next results for tuples of selfadjoint operators.

Our arguments rely essentially on the following result describing big subsets of the numerical range of a tuple in terms
of its essential numerical range.

\begin{proposition}\label{P4.1}
Let $\cS=(S_1,\dots,S_s)\in B(H)^s$ be an $s$-tuple of selfadjoint operators. Then
$$
\Int_{\RR^s} W_{\rm e}(\cS)\subset W_\infty(\cS).
$$
\end{proposition}

\begin{proof}
The proof of the proposition is exactly the same as the proof of  \cite[Corolary 4.3]{Muller18}, where it was proved that $\Int W_{\rm e}(\cT)\subset W_\infty(\cT)$ for any $n$-tuple $\cT=(T_1,\dots,T_n)\in B(H)^n,$ the interior being considered in $\CC^n$. Thus, we omit the arguments.
\end{proof}

The proof of Theorem \ref{T1.1}, the main result of this section,
introduces a technique which will be crucial for the whole of paper.
\bigskip

\emph{Proof of Theorem \ref{T1.1}}\,\,\,
The assumptions of Theorem \ref{T1.1} imply that $S_j\ne 0$ for at least one of $j=1,\dots,s$.

Let $(y_m)_{m=1}^\infty$ be a sequence of unit vectors in $H$ such that $\bigvee_my_m=H$. Using \eqref{blaschke}, we can find mutually
disjoint sets $A_m, m\in\NN,$ such that $\bigcup_{m=1}^\infty A_m=\NN$ and
$$
\sum_{k\in A_m}\dist\{\la_k,\RR^s\setminus W_{\rm e}(\cS)\}=\infty
$$
for all $m\in\NN$.

It is sufficient to construct an orthonormal sequence $(u_k)_{k=1}^{\infty}$ in $H$ such that $\langle \cS u_k,u_k\rangle=\la_k$ for all $k\in\NN$ and
$$
\ln\dist^2\Bigl\{y_m,\bigvee_{k\le N}u_k\Bigr\}\le
-\frac{\sum_{k\le N,k\in A_m}\dist\{\la_k,\RR^s\setminus W_{\rm e}(\cS)\}}{4\max\{\|S_1\|,\dots,\|S_s\|\}}
$$
for all $N,m\in\NN$.
Indeed, since $$\sum_{k\in A_m}\dist\{\la_k,\RR^s\setminus W_{\rm e}(\cS)\}=\infty$$ for all $m\in\NN$, we will then have
$$
\lim_{N\to\infty}\dist\Bigl\{y_m,\bigvee_{k\le N}u_k\Bigr\}=0,
$$
so $y_m\in\bigvee_{k\in\NN}u_k$ for all $m$, and so $(u_k)_{k=1}^{\infty}$
will be an orthonormal basis satisfying $\langle \cS u_k,u_k\rangle=\la_k$ for all $k\in\NN$.

The sequence $(u_k)_{k=1}^{\infty}$ will be constructed inductively.
Set formally $u_0=0,$ let $N\in\mathbb N,$ and if $N\ge 2$ suppose that $u_k, 1 \le k \le {N-1},$ is an orthonormal set  satisfying
$\langle \cS u_k,u_k\rangle=\la_k$ for all $k\le N-1$ and
$$
\ln\dist^2\Bigl\{y_m,\bigvee_{k\le N-1}u_k\Bigr\}\le
-\frac{\sum_{k\le N-1,k\in A_m}\dist\{\la_k,\RR^s\setminus W_{\rm e}(\cS)\}}{4\max\{\|S_1\|,\dots,\|S_s\|\}}
$$
for all $m\in\NN$.

Write $M_{N-1}=\bigvee_{k\le N-1}u_k$, and let $m \in \mathbb N$ be such that $N\in A_m$. We are looking
for a unit vector $u_{N}\in M_{N-1}^\perp$ such that $\langle \cS u_{N},u_{N}\rangle=\la_{N}$ and
$$
\ln\dist^2\Bigl\{y_m,\bigvee_{k\le N}u_k\Bigr\}\le
-\frac{\sum_{k\le N,k\in A_m}\dist\{\la_k,\RR^s\setminus W_{\rm e}(\cS)\}}{4\max\{\|S_1\|,\dots,\|S_s\|\}}.
$$
By Proposition \ref{P4.1}  one has $\la_{N}\in\Int W_{\rm e}(\cS)\subset W_\infty(\cS).$ Hence
if $y_m\in M_{N-1}$ then it suffices to take any unit vector $u_{N}\in M_{N-1}^\perp$ satisfying
$\langle \cS u_N,u_{N}\rangle=\la_{N}.$

Suppose that $y_m\notin M_{N-1}$. Then
 \begin{equation}\label{y}
 y_m=a+tb
 \end{equation}
 with
\begin{equation}\label{rep}
 a\in M_{N-1}, \qquad b\perp M_{N-1}, \|b\|=1 \quad \text{and} \quad t=\dist\{y_m,M_{N-1}\}>0,
\end{equation}
and so
$$
\ln t^2\le
-\frac{\sum_{k\le N-1,k\in A_m}\dist\{\la_k,\RR^s\setminus W_{\rm e}(\cS)\}}{4\max\{\|S_1\|,\dots,\|S_s\|\}}.
$$

If $\langle \cS b,b\rangle=\la_{N}$ then set $u_{N}:=b$.
If $\langle \cS b,b\rangle\ne\la_{N}$, then let
$$
\rho=\bigl\|\langle \cS b,b\rangle-\la_{N}\bigr\|
\qquad
\text{and}
\qquad
\de=\frac{1}{2}\dist\{\la_{N},\RR^s\setminus W_{\rm e}(\cS)\}.
$$

Let $\mu \in \mathbb C^s$ be the unique vector
such that $\|\mu-\la_{N}\|=\de$ and
$$
\frac{\langle \cS b,b\rangle-\la_{N}}{\rho}=\frac{\la_{N}-\mu}{\de}.
$$
Clearly, $\mu\in\Int W_{\rm e}(\cS)\subset W_\infty(\cS)$ by the choice of $\delta$ and Theorem \ref{wint}.

The subspace generated by $M_{N-1}$ and $b, S_1b,\dots,S_s b$ is  finite-dimensional in $H.$
 So there exists $x\in H$, $\|x\|=1$ such that $$x\perp M_{N-1}, b, S_1b,\dots,S_s b, \qquad \text{and}
\qquad \langle \cS x,x\rangle=\mu,$$
and we set
$$
u_{N}:=\sqrt{\frac{\rho}{\rho+\de}}\,x+\sqrt{\frac{\de}{\rho+\de}}\,b.
$$

We have $\|u_{N}\|=1$ and $u_{N}\perp M_{N-1},$ and moreover
\begin{align*}
\langle \cS u_{N},u_{N}\rangle=&
\frac{\rho}{\rho+\de}\,\mu+\frac{\de}{\rho+\de}\,\langle \cS b,b\rangle\\
=&\frac{\rho}{\rho+\de}\,(\mu-\la_{N})+\frac{\de}{\rho+\de}\,(\langle \cS b,b\rangle-\la_{N})+\la_{N}\\
=&\la_{N}.
\end{align*}

It remains to estimate the distance $\dist\{y_m,M_{N}\}$.
In view of \eqref{y} and \eqref{rep},
taking into account that $\langle b,u_{N}\rangle=\sqrt{\frac{\de}{\rho+\de}}$,
we have
$$
\dist^2\{y_m,M_{N}\}=
\dist^2\{tb,M_{N}\}=t^2\cdot\dist^2\Bigl\{b,\bigvee_{k=1}^{N} u_ k\Bigr\}=
t^2\Bigl(1-\frac{\de}{\rho+\de}\Bigr)
$$
and
\begin{align*}
\ln\frac{\dist^2\{y_m,M_{N}\}}{\dist^2\{y_m,M_{N-1}\}}=&
\ln\Bigl(1-\frac{\de}{\rho+\de}\Bigr)
\le -\frac{\de}{\rho+\de}\\
\le&
-\frac{\dist\{\la_{N},\RR^s\setminus W_{\rm e}(\cS)\}}{4\max\{\|S_1\|,\dots,\|S_s\|\}}.
\end{align*}
Hence
\begin{align*}
\ln\dist^2\{y_m,M_{N}\}\le&
\ln\dist^2\{y_m,M_{N-1}\}-\frac{\dist\{\la_{N},\RR^s\setminus W_{\rm e}(\cS)\}}{4\max\{\|S_1\|,\dots,\|S_s\|\}}\\
\le&
-\sum_{k\le N, k\in A_m} \frac{\dist\{\la_k,\RR^s\setminus W_{\rm e}(\cS)\}}{4\max\{\|S_1\|,\dots,\|S_s\|\}}.
\end{align*}
This finishes the proof. \hfill $\Box$

\begin{corollary}\label{complex}
Let $\cT=(T_1,\dots,T_n)\in B(H)^n$. For $W_{\rm e}(\cT)$ identified with a subset of $\mathbb R^{2n},$ let $M\subset\mathbb R^{2n}$ be the smallest affine  subspace containing $W(\cT)$. Let $(\la_k)_{k=1}^{\infty} \subset\Int_M W_{\rm e}(\cT)$ satisfy
$$
\sum_{k=1}^\infty \dist\{\la_k,M\setminus W_{\rm e}(\cT)\}=\infty.
$$
Then $(\la_k)_{k=1}^{\infty}\in\mathcal D(\cT)$.
\end{corollary}

\begin{proof}
It is sufficient to consider tuples of selfadjoint operators. So assume that $\cS=(S_1,\dots,S_s)\in B(H)^s$ is an $s$-tuple of selfadjoint operators. Let $M$ be the smallest affine subspace containing  $W(\cS),$ and let $(\la_k)_{k=1}^{\infty}\subset\Int_M W_{\rm e}(\cS)$ satisfy
$\sum_{k=1}^\infty \dist\{\la_k,M\setminus W_{\rm e}(\cS)\}=\infty$.

We prove the statement by induction on $s$. If $s=1$, then either $M=\RR$ and the statement follows from the previous theorem, or $M$ is a single point, $M=\{t\}$, so that $S_1=tI$ and there's nothing to prove.

Suppose that the statement is true for $s\ge 1$. Let $\cS=(S_1,\dots,S_{s+1})$ be an $(s+1)$-tuple of selfadjoint operators. Let $M$ be the affine subspace generated by $W(\cS)$. If $M=\RR^{s+1}$ then the statement follows from the previous theorem.

Suppose that $M\ne\RR^{s+1}$. Then there exist $\al_0,\al_1,\dots,\al_{s+1}\in\RR$ such that $(\al_1,\dots,\al_{s+1})\ne(0,\dots,0)$ and for all $m=(m_1,\dots,m_{s+1})\in M$,
$$
\al_0+\sum_{j=1}^{s+1}\al_j m_j=0.
$$
Without loss of generality we may assume that $\al_{s+1}\ne 0$. Then
$$
W\Bigl(\frac{\al_0}{\al_{s+1}}I+\sum_{j=1}^{s+1}\frac{\al_j}{\al_{s+1}} S_j\Bigr)=\{0\},
$$
and so
\begin{equation}\label{star}
S_{s+1}=-\frac{\al_0}{\al_{s+1}}I-\sum_{j=1}^{s}\frac{\al_j}{\al_{s+1}} S_j.
\end{equation}
Let $\cS'=(S_1,\dots,S_s)$ and let $P:\RR^{s+1}\to\RR^s$ be the natural projection onto the first $s$ coordinates.
For $k\in\NN$ let $\la_k=(\la_{k,1},\dots,\la_{k,s+1})$. By \eqref{star}, we have
$$
\la_{k,s+1}=-\frac{\al_0}{\al_{s+1}}-\sum_{j=1}^s\frac{\al_j}{\al_{s+1}}\la_{k,j}
$$
for all $k\in\NN$. So it is sufficient to show that $(P\la_k)_{k=1}^{\infty}\in{\mathcal D}(\cS')$.

The smallest affine subspace containing $W(\cS')$ is $P(M)$ and it is easy to see that
$\sum_{k=1}^\infty\dist\{P\la_k,P(M)\setminus W_{\rm e}(\cS')\}=\infty$. By the induction assumption this implies that $(P\la_k)_{k=1}^{\infty}\in{\mathcal D}(\cS')$.
\end{proof}

\begin{corollary}\label{C4.4}
Let $\cT=(T_1,\dots,T_n)\in B(H)^n$. Let $(\la_k)_{k=1}^{\infty}\subset\Int W_{\rm e}(\cT)$ satisfy
\begin{equation}\label{blaschke1}
\sum_{k=1}^\infty \dist\{\la_k, \mathbb C^n \setminus W_{\rm e}(\cT)\}=\infty.
\end{equation}
Then $(\la_k)_{k=1}^\infty\in\mathcal D(\cT)$.
\end{corollary}

\begin{remark}\label{diagon_rem}
The corollary above, and even weaker Herrero's result mentioned in the introduction,
 allows one to deduce directly Fong's theorem from \cite{Fong87}
saying that any bounded sequence in $\mathbb C$ can be realized as a diagonal of square zero nilpotent $N$.
As noted in \cite[p. 864]{Herrero91}, it is enough to observe that $W_{\rm e}(N)$  can contain an arbitrary disc (with center at zero)
for an appropriate $N.$
Similarly, using \cite[Corollary 2.6]{Li-Rocky}  (and Theorem 1.1 quoted there and proved in \cite{Tse-Rocky}),
one can prove that any bounded sequence from $\mathbb C$ can be realized as a diagonal of idempotent thus recovering
\cite[Theorem 3.7]{Loreaux16}. In fact, using \cite[Corollary 2.6]{Li-Rocky}, similar results can be proved for more general
quadratic operators, but we refrain ourselves from giving straightforward details.
\end{remark}

 A natural question is
  how fast a sequence should approach the boundary of the essential numerical range to be realizable as a diagonal, and whether the assumption \eqref{blaschke} is optimal. It appears that  \eqref{blaschke} cannot in general be improved  as a simple Herrero's example  from \cite[p. 862-863]{Herrero91} shows. Namely, it was proved in \cite{Herrero91} that if $T$ is the unilateral shift, then $W_{\rm e}(T)=\overline{\mathbb D}$, and a sequence $(\la_k)_{k=1}^{\infty}\subset\DD$ is a diagonal of $T$ if and only if $\sum_{k=1}^\infty(1-|\la_k|)=\infty$.

The same proof works in a more general setting, see also \cite[Lemma 4.1]{Loreaux18} for a similar argument.

\begin{proposition}
Let $T\in B(H)$ be such that $\|T\|\le 1$ and $W_{\rm e}(T)=\overline\DD$. Suppose that $T$ is not a Fredholm operator of index $0$. Let $(\la_k)_{k=1}^\infty\subset\DD$. Then
$$
(\la_k)_{k=1}^{\infty}\in\mathcal D(T) \Longleftrightarrow \sum_{k=1}^\infty (1-|\la_k|)=\infty.
$$
\end{proposition}

Unfortunately, Blaschke-type conditions \eqref{blaschke} and \eqref{blaschke1} are  only sufficient for a sequence $(\la_k)_{k=1}^{\infty}$ to be in $\mathcal D(\cT)$ for an $n$-tuple $\cT\in B(H)^n$.
On the other hand, we can also formulate a necessary condition as well. Let us start with a single selfadjoint operator.

\begin{proposition}
Let $T\in B(H)$, $T\ge 0$. Let $(\la_k)_{k=1}^{\infty}\subset\Int_\RR W_{\rm e}(T)$ and $(\la_k)_{k=1}^{\infty}\in \mathcal D(T)$. Then
$\sum_{k=1}^\infty \la_k=\infty$.
\end{proposition}

\begin{proof}
Suppose that $(\la_k)_{k=1}^{\infty} \in\Int_\RR W_{\rm e}(T)$, $(\la_k)_{k=1}^{\infty} \in\mD(T)$ and $\sum_{k=1}^\infty\la_k<\infty$. Then there exists an orthonormal basis $(u_k)_{k=1}^{\infty}$ in $H$ such that $\langle Tu_k,u_k\rangle=\la_k$ for all $k\in\NN$. We have
$$
\sum_{k=1}^\infty\|T^{1/2}u_k\|^2=
\sum_{k=1}^\infty\la_k<\infty.
$$
So $T^{1/2}$ is a Hilbert-Schmidt operator. Thus $T$ is compact, $W_{\rm e}(T)=\{0\}$ and $\Int_{\RR} W_{\rm e}(T)=\emptyset$, a contradiction.
\end{proof}

\begin{corollary}
Let $\cS=(S_1,\dots,S_s)\in B(H)^s$ be an $s$-tuple of selfadjoint operators, let
$(\la_k)_{k=1}^{\infty}\subset {\rm Int}_{\RR^s}\, W_{\rm e}(\cS)$, $(\la_k)_{k=1}^{\infty}\in \mathcal D(\cS),$
 and $\la_k=(\la_{k,1},\dots,\la_{k,s})$ for all $k\in\NN$. Let $\al_0,\dots,\al_s$ be real numbers, $(\al_1,\dots,\al_s)\ne(0,\dots,0)$, $V=\al_0 I+\sum_{j=1}^s\al_j S_j$ and $a=\inf\{t: t\in W (V)\}$. Then
$$
\sum_{k=1}^\infty \Bigl(\al_0-a+\sum_{j=1}^s\al_j\la_{k,j}\Bigr)=\infty.
$$
\end{corollary}

\begin{proof}
From our assumptions it follows that
$$
\Bigl(\al_0-a+\sum_{j=1}^s\al_j\la_{k,j}\Bigr)\in
\mathcal D\Bigl(V-a\Bigr),
$$
$$\Bigl(\al_0-a+\sum_{j=1}^s\al_j\la_{k,j}\Bigr)\in\Int_\RR W_{\rm e}(V-a)$$
 for all $k\in\NN,$
and
$V-a \ge 0.$ So the statement follows from the previous proposition.
\end{proof}

In general, despite the Blaschke-type conditions identify a subset of $\mathcal D (\mathcal T)$
they are far from being characterizations of the whole set $\mathcal D (\mathcal T)$
even if ${\rm Int}\, W_{\rm e}(\mathcal T)$ is large.
Examples of selfadjoint projections (Kadison's theorem) and of normal operators with finite spectrum
(Arveson's obstruction theorem)
can serve as simple illustrations of this fact.

While numerical ranges are useful for dealing with operator diagonals (as the theorem above confirms),
it is also of interest to relate the spectral structure of an operator to its set of diagonals,
thus showing an unexpected link between both.
This task appears to be more demanding: while diagonals ``live'' in the numerical range, their relation to spectrum is far from being obvious.
On this way, we prove a result describing a part of diagonals for powers of an operator  by means
of the polynomial hull of its  spectrum.

We start with several auxiliary statements.

\begin{lemma}\label{L1.2}
Let $\rho>0$ and $n\in\NN$. Then there exists $b_n>0$ with the following property:
If $\e=(\e_1,\dots,\e_n)\in\CC^n$ then there are  $s\in\NN$, $\mu_1,\dots,\mu_s\in \TT_\rho$ and $\al_1,\dots,\al_s\ge 0$ such that
$$
\sum_{j=1}^s \al_j\le b_n\|\e\| \qquad \hbox{and} \qquad \sum_{j=1}^s \al_j\mu_j^k=\e_k, \qquad 1 \le k \le n.
$$
More precisely, one can choose
$$b_n=\begin{cases}(2^{n}-1)\rho^{-n},& \quad  \rho\le 1,\\
2^{n}-1,& \quad \rho \ge 1.
\end{cases}
$$
\end{lemma}

\begin{proof}
We prove the lemma  by induction on $n$.
For $n=1$ set $b_1=\rho^{-1}$.
If $\e_1\in\CC$, $\e_1=|\e_1|\cdot e^{2\pi i \varphi}$ for some $\varphi\in [0,1)$, then take $\la_1=\rho e^{2\pi i\varphi}$ and $\al_1=\frac{|\e_1|}{\rho}$. We have
$$
\al_1\la_1=|\e_1|\cdot  e^{2\pi i\varphi}=\e_1 \qquad \hbox{and} \qquad \al_1=b_1|\e_1|,$$
hence the lemma clearly holds for $n=1$ and $b_1=\rho^{-1}$.

Suppose that the lemma is true for some integer $n-1\ge 1$, and prove it for $n$.
Set $$b_n=2b_{n-1}+\rho^{-n}.$$ By the induction assumption,
there exist $l\in\NN$, $z_1,\dots, z_l\in \TT_\rho$ and $\alpha_1,\dots,\alpha_l\ge 0$ such that
$$\sum_{j=1}^{l}\alpha_j\le b_{n-1}\|\e\|$$
and
$$
\sum_{j=1}^{l} \alpha_j z^{k}_j=\e_k, \qquad k=1,\dots,n-1.
$$
Let
$$\tilde\e_n=\e_n-\sum_{j=1}^{l}\alpha_j z^{n}_j.$$
Then
$$|\tilde\e_n|\le|\e_n|+\rho^n\sum_{j=1}^{l} \alpha_j\le\|\e\|+\rho^nb_{n-1}\|\e\|.$$
Write $\tilde\e_n=|\tilde\e_n|\cdot e^{2\pi i\varphi}$ for some $\varphi \in [0,1)$ and
set
$$\xi_j=\rho e^{2\pi i(\varphi+\frac{j}{n})} \qquad \hbox{and} \qquad \beta_j=\frac{|\tilde\e_n|}{n\rho^n}, \qquad 1 \le j \le n.$$

If $1 \le k \le n-1$ then
$$
\sum_{j=1}^n \beta_j\xi^{k}_j=
\frac{|\tilde\e_n|}{n\rho^n} \rho^ke^{2\pi ik \varphi}\sum_{j=1}^{n} e^{2\pi ijk/n}=0.
$$
Similarly,
$$
\sum_{j=1}^n \beta_j\xi^{n}_j=
\frac{|\tilde\e_n|}{n\rho^n} \rho^n e^{2\pi i\varphi}\cdot n=\tilde\e_n.
$$

Thus for every $k,$ $ 1\le k \le n-1,$ we have
$$
\sum_{j=1}^{l} \alpha_jz^{k}_j+\sum_{j=1}^n \beta_j\xi^{k}_j=\e_k
$$
and
$$
\sum_{j=1}^{l} \alpha_jz^{n}_j+\sum_{j=1}^n \beta_j\xi^{n}_j=
\e_n-\tilde\e_n+\tilde\e_n=\e_n.
$$
Finally,
$$
\sum_{j=1}^{l} \alpha_j+\sum_{j=1}^n \beta_j\le
b_{n-1}\|\e\|+\frac{|\tilde\e_n|}{\rho^n}\le
\|\e\|\bigl(2b_{n-1}+\rho^{-n}\bigr)=b_n\|\e\|.
$$

For $\rho\le 1$ one can prove easily by induction that $b_n\le (2^{n}-1)\rho^{-n}$. Similarly, for $\rho\ge 1$ one has $b_n\le 2^{n}-1$.

\end{proof}
For $\lambda \in \mathbb C, \rho >0$ and $n \in \mathbb N$ let $C_{\rho, \lambda}:= \conv\Bigl \{(\mu,\dots,\mu^n):|\mu-\la|\le\rho\Bigr \}.$

\begin{lemma}\label{P1.3}
Let $\la\in\CC$, $n\in\NN$ and $0<\rho\le 1$. Then
$$
\dist \{(\la,\dots,\la^n), \partial C_{\rho, \lambda}\}
\ge\frac{\rho^n}{4^n\max\{1,|\la|^n\}}.
$$
If $\la=0$ then
$$
\dist \{(0,\dots,0), \partial C_{\rho, 0}\}
\ge\frac{\rho^n}{2^n}.
$$
\end{lemma}

\begin{proof}
First, assume that $\lambda =0.$  Let $\e=(\e_1,\dots,\e_n)\in\CC^n$ be such that $\|\e\|\le\rho^n 2^{-n}.$
By Lemma \ref{L1.2}, there exist $s\in\NN$, $\mu_1,\dots,\mu_s\in\CC$, with $|\mu_j|=\rho$ for all $j=1,\dots,s$ and $\al_1,\dots,\al_s\ge 0$ with $$\sum_{j=1}^s\al_j\le (2^n-1)\rho^{-n}\|\e\|<1$$ such that
$$
\sum_{j=1}^s \al_j\mu_j^k=\e_k, \qquad k=1,\dots,n.
$$
Thus
$$
\e_k=\sum_{j=1}^s \al_j\mu_j^k+\Bigl(1-\sum_{j=1}^s\al_j\Bigr)\cdot 0
$$
and
$$
\e\in\conv\Bigl\{(0,\dots,0),(\mu_j,\dots,\mu^n_j): j=1,\dots,s\Bigr\}=C_{\rho,0}.
$$
In other words,
$$
\dist \{(0,\dots,0), \partial C_{\rho, 0} \}
\ge\frac{\rho^n}{2^n},
$$
so that the lemma holds for $\lambda=0.$

Let now  $\la\in\CC$ be arbitrary. Consider the affine mapping
$G_\la:(\CC^n, \|\cdot\|) \to (\CC^n, \|\cdot\|)$ defined by
$$
G_\la(z_1,\dots,z_n)= \Bigl(z_1+\lambda, z_2+2\lambda z_1+\lambda^2, \dots ,
\sum_{j=1}^n {n\choose j} z_j\lambda^{n-j} + \lambda^n \Bigr).
$$
Then  the mapping $H_\la: (\CC^n, \|\cdot\|) \to (\CC^n \|\cdot\|)$ given by
$$H_\la(z_1,\dots,z_n)=
G_\la(z_1,\dots,z_n)-(\lambda,\lambda^2,\dots,\lambda^n)$$
is linear and invertible (since it is determined by an upper triangular
matrix with non-zero diagonal). Clearly $\|H_\la\|\le 2^n\max\{1,|\la|^n\}$. So
\begin{equation}\label{glam}
\|G_\la u-G_\la v\|\le 2^n\max\{1,|\la|^n\}\cdot\|u-v\|
\end{equation}
for all $u,v\in\CC^n$. Moreover, for each $\mu\in\CC$ we have
$$
G_\la(\mu,\dots,\mu^n)=(\la+\mu,\dots,(\la+\mu)^n).
$$
  In view of $G_{-\la}G_\la(0,\dots,0)=(0,\dots,0)$, the mapping $G_{-\la}G_\la$ is linear and satisfies $G_{-\la}G_\la(\mu,\dots,\mu^n)=(\mu,\dots,\mu^n)$ for all $\mu\in\CC$. By considering appropriate Vandermonde determinants, it follows that the linear span of $\{(\mu,\dots,\mu^n): \mu\in\CC\}$ is the whole of  $\CC^n.$ So  $G_{-\la}G_\la$ is the identity on $\CC^n$.
Using \eqref{glam}, we infer that
$$
\|u-v\|\le \|G_\la u- G_\la v\|\cdot 2^n\max\{1,|\la|^n\}
$$
for all $u,v$. Finally,
\begin{align*}
&\dist \{(\la,\dots,\la^n), \partial C_{\rho,\lambda}\}\\
=&
\dist\Bigl\{G_\la(0,\dots,0), \partial \, \conv\{G_\la(\mu,\dots,\mu^n):|\mu|\le\rho\}\Bigr\}\\
\ge&
\dist \{(0,\dots,0), \partial C_{\rho, 0}\}\cdot \frac{1}{2^{n}\max\{1,|\la|^n\}}\\
\ge&
\frac{\rho^n}{4^n\max\{1,|\la|^n\}}.
\end{align*}
\end{proof}
The following corollary of independent interest will be needed in the next section.
\begin{corollary}\label{C4.10}
Let $T\in B(H)$, $0\in\Int\widehat\sigma(T)$, and
$\rho:=\dist\{0,\partial\widehat\sigma(T)\}\le 1$. Let $\mu\in\CC^n$ and
$\|\mu\|<2^{-n}\rho^n$. Then $\mu\in W_\infty(T,\dots,T^n)$.
\end{corollary}

\begin{proof}
 By Lemma \ref{P1.3} and Theorem \ref{lambdawe}, we have
$$\mu\in \Int C_{\rho,0}\subset \Int W_e(T,\dots,T^n)\subset W_\infty(T,\dots,T^n).$$
\end{proof}

Now we are ready to link spectral properties of operator tuples to their sets of diagonals.

\begin{theorem}
Let $T\in B(H)$, $n\in\NN$, $(\la_k)_{k=1}^\infty\subset\Int\widehat\si(T)$. Suppose that $\sum_{k=1}^\infty \dist^n\{\la_k,\partial\widehat\si(T)\}=\infty$. Then there exists an orthonormal basis $(u_k)_{k=1}^{\infty}$ in $H$ such that
$$
\langle T^ju_k,u_k\rangle=\la_k^j, \qquad k\in\NN,1\le j\le n.
$$
\end{theorem}

\begin{proof}
Let $\cT_n=(T,T^2,\dots,T^n)$.
Since $W_{\rm e}(\mathcal T_n)$ is convex, by Theorem \ref{lambdawe},
$$
W_{\rm e}(\cT_n)\supset\conv\{(\la,\dots,\la^n):\la\in\Int\widehat\si(T)\}.
$$

By Lemma \ref{P1.3}, $(\la_k,\dots,\la_k^n)\in\Int W_{\rm e}(\cT_n)$ for all $k$, and
$$
\sum_{k=1}^\infty \dist\bigl\{(\la_k,\la_k^2,\dots,\la_k^n),\partial W_{\rm e}(\cT_n)\bigr\}=\infty,
$$
So the statement follows from Corollary \ref{C4.4}.
\end{proof}

\section{Diagonals: compact perturbations}

Let $T\in B(H)$ and $0\in W_{\rm e}(T)$. In \cite[Theorem 2.3]{Stout81},
Q. Stout showed that for each sequence $(\al_k)_{k=1}^{\infty}\notin\ell_1$
there exists an orthonormal basis $(u_k)_{k=1}^{\infty}$ in $H$ such that the corresponding diagonal $\langle Tu_k,u_k\rangle_{k=1}^{\infty}$ of $T$ satisfies
$$
|\langle Tu_k,u_k\rangle|\le|\al_k|
$$
for all $k\in\NN$. In particular, for each $p>1$ it is possible to construct a diagonal of $T$
satisfying $\sum_{k=1}^\infty |\langle Tu_k,u_k\rangle|^p<\infty$.
This is an older result due to Anderson \cite{AndersonThesis},
see also \cite[Theorem 4.1]{Fong87}.

By the techniques of this paper, we improve Stout's result in Theorem \ref{T2.1} in two directions.
First, we show that any sequence $(\la_k)_{k=1}^{\infty}\subset W_{\rm e}(T)$ can be approximated by a diagonal in the sense of
\eqref{stout}
 (Stout's statement treats just zero sequences) and, second, we are able to obtain the result for tuples of operators,
rather than for a single operator. Note that Theorem \ref{T2.1} generalizes  also the corresponding Herrero's result from \cite{Herrero91}.

\bigskip
\emph{Proof of Theorem \ref{T2.1}}\,\,\,
We argue as in the proof of Theorem \ref{T1.1}, though the technical details deviate at several points.

Let $(y_m)_{m=1}^\infty$ be a sequence of unit vectors in $H$ such that $\bigvee_my_m=H$.
We can find mutually disjoint sets $A_m,  m\in\NN,$ such that $\bigcup_{m=1}^\infty A_m=\NN$ and
$$
\sum_{k\in A_m}|\al_k|=\infty
$$
for all $m\in\NN$.

Again it is sufficient to construct an orthonormal sequence $(u_k)_{k=1}^{\infty}$ in $H$ such that $\|\langle \cT u_k,u_k\rangle-\la_k\|\le|\al_k|$ for all $k\in\NN$ and
$$
\ln\dist^2\Bigl\{y_m,\bigvee_{k\le N}u_k\Bigr\}\le
-\frac{\sum_{k\le N,k\in A_m}|\al_k|}{4\max\{\|T_1\|,\dots,\|T_n\|\}}
$$
for all $N,m\in\NN$.

The sequence $(u_k)_{k=1}^{\infty}$ will be constructed inductively.
Set formally $u_0=0,$ let $N \in \mathbb N,$ and if $N \ge 2$ suppose that $u_1,\dots,u_{N-1}$ is an orthonormal set satisfying
$\|\langle \cT u_k,u_k\rangle-\la_k\|\le|\al_k|$ for all $k\le N-1$ and
$$
\ln\dist^2\Bigl\{y_m,\bigvee_{k\le N-1}u_k\Bigr\}\le
-\frac{\sum_{k\le N-1,k\in A_m}|\al_k|}{4\max\{\|T_1\|,\dots,\|T_n\|\}}
$$
for all $m\in\NN$. Write $M_{N-1}=\bigvee_{k\le N-1}u_k$ and let $N\in A_m$ for some $m \in \mathbb N.$

Decompose $y_m$ as  $y_m=a+tb$, where $a\in M_{N-1}$, $b\perp M_{N-1}$, $\|b\|=1$ and $t=\dist\{y_m,M_{N-1}\}$.

Using $\lambda_N \in W_{\rm e}(\cT),$ choose a unit vector $v\in M_{N-1}^\perp$ such that
$$
v\perp b, T_1 b,T_1^*b,\dots,T_n b,T_n^* b, \quad \text{and} \quad \bigl\|\langle \cT v,v\rangle-\la_{N}\bigr\|\le \frac{|\al_{N}|}{2}
$$
and set
$$
u_{N}=cb+\sqrt{1-c^2}\,v,
$$
where
$$c=\min\Bigl\{1,\Bigl(\frac{|\al_{N}|}{4\max\{\|T_1\|,\dots,\|T_n\|\}}\Bigr)^{1/2}\Bigr\}.$$
Clearly $\|u_{N}\|=1$ and $u_{N}\perp u_1,\dots,u_{N-1}$.

We have
\begin{align*}
\bigl\|\langle\cT u_{N},u_{N}\rangle-\la_{N}\bigr\|\le&
\bigl\|\langle\cT u_{N},u_{N}\rangle-\langle \cT v,v\rangle\bigr\|+ \bigl\|\langle\cT v,v\rangle-\la_{N}\bigr\|\\
\le&
c^2\|\langle \cT b,b\rangle\|+ c^2\|\langle\cT v,v\rangle\|+\frac{|\al_{N}|}{2}\\
\le&
2c^2\max\bigl\{\|T_1\|,\dots,\|T_n\|\bigr\}+\frac{|\al_{N}|}{2}\\
\le& |\al_{N}|.
\end{align*}
Furthermore,
$$
\dist^2\Bigl\{y_m,\bigvee_{k\le N}u_k\Bigr\}\le
t^2\dist^2\Bigl\{b,\bigvee_{k \le N} u_{N}\Bigr\}
=t^2(1-c^2).
$$
Thus
\begin{align*}
\ln\dist^2\Bigl\{y_m,\bigvee_{k\le N}u_k\Bigr\}\le&
\ln t^2+ \ln(1-c^2)\le \ln t^2 -c^2\\
\le& \ln t^2- \frac{|\al_{N}|}{4\max\{\|T_1\|,\dots,\|T_n\}}\\
\le&
-\frac{\sum_{k\le N,k\in A_m}|\al_k|}{4\max\{\|T_1\|,\dots,\|T_n\|\}}.
\end{align*}
\hfill $\Box$

Now, by means of  Theorem \ref{T2.1}, we can describe
the set of diagonals $\mathcal D (\cT)$
up to $p$-Schatten class perturbations of $\cT$.
In this more general setting, we are able to construct the diagonals
satisfying a generalized Blaschke-type condition
resembling in a sense the definition of $S_p$-classes.
Moreover, the diagonals of perturbations may not necessarily
belong to $W_{\rm e}(\cT)$ but should only approximate $W_{\rm e}(\cT)$
good enough, where the rate of approximation is determined
by the Schatten class of perturbations.

\begin{corollary}\label{C2.2}
Let $\cT=(T_1,\dots,T_n)\in B(H)^n$. Let $p> 1$. Let $(\la_k)_{k=1}^{\infty}\subset \CC^n$ satisfy
\begin{equation}\label{schatten_p}
\sum_{k=1}^\infty \dist^p\{\la_k,W_{\rm e}(\cT)\}<\infty.
\end{equation}
Then there exists an $n$-tuple of operators $\cK=(K_1,\dots,K_n)$ with $K_j$  from the Schatten class $S_p,$ $1\le j\le n,$ such that
$(\la_k)_{k=1}^{\infty} \in \mathcal D(\cT +\cK).$
\end{corollary}

\begin{proof}
Let $p >1$ be fixed and let $(\la_k)_{k=1}^{\infty}\subset\CC^n$ satisfy \eqref{schatten_p}.

Find a sequence $(\la_k')_{k=1}^{\infty}\subset W_{\rm e}(\cT)$ such that $\sum_{k=1}^\infty\|\la_k'-\la_k\|^p< \infty$. By Theorem \ref{T2.1}, there exists an orthonormal basis $(u_k)_{k=1}^{\infty}$ in $H$ such that $$\|\langle\cT u_k,u_k\rangle-\la_k'\|\le k^{-1}$$ for all $k$.
Define $\cK=(K_1,\dots,K_n)\in B(H)^n$ as
$$
\cK u_k=(\la_k-\langle\cT u_k,u_k\rangle)u_k
$$
for all $k$. Then
$$
\langle(\cT+\cK)u_k,u_k\rangle=\la_k, \qquad k\in\NN,
$$
and
$$
\sum_{k=1}^\infty \|\la_k-\langle\cT u_k,u_k\rangle\|^p
\le
\sum_{k=1}^\infty \Bigl(\|\la_k-\la_k'\|+\|\la_k'-\langle\cT u_k,u_k\rangle\|\Bigr)^p
<\infty.
$$
So the operators $K_1,\dots,K_n$ belong to $S_p$.
\end{proof}

We finish this section with a discussion of a subset
$
\mD_{\rm const} (\cT)=\{
\la\in\CC^n: (\la,\la,\dots)\in \mD (\cT)\}
$
of diagonals of $T \in B(H)$
that seems to be crucial.
Recall that operators possessing a zero diagonal appeared relevant in several areas of operator theory
(e.g. the study of commutators) and
attracted a substantial attention much before the foundational works of Kadison and Arveson on the set of (all) diagonals.
Thus it is natural to try to understand the whole set of constant diagonals for a fixed $T$
and to relate its structure to the structure of $\sigma(T)$ and $W(T).$
While this task seems to be more accessible than a characterization of $\mathcal T,$
we are far from a complete answer even for very simple operators $T.$

Observe that by Theorem \ref{T1.1},
\begin{equation}\label{inclusion_e}
\Int W_{\rm e}(\cT)\subset \mathcal D_{\rm const}(\cT)\subset W_{\rm e}(\cT)\cap W(\cT).
\end{equation}

Since the interior of a convex set is convex, both $W_{\rm e}(\cT)$ and $\Int W_{\rm e}(\cT)$ are convex sets.
However, the question whether $\mathcal D_{\rm const}(\cT)$ is convex is still open, even if $n=1$.
This problem has been posed explicitly in \cite[p. 213]{Bourin03}.
It is instructive to note that, even if $n=1$, the diagonal set $\mathcal D_{\rm const}(T)$ is not a subset of $W_{\infty}(T),$ and even the inclusion $\mathcal D_{\rm const}(T)\subset W_2(T)$ may not hold, where
  $W_2(\cT)$ stands for  the set of all $\la \in \mathbb C$ such that there exists a two-dimensional subspace $L\subset H$ satisfying
$
P_LTP_L=\la P_L.
$
The next example makes this precise.

\begin{example}
Let  $T=\diag\bigl(-1,\frac{1}{2},\frac{1}{4},\frac{1}{8},\dots\bigr)$.
By Fan's result \cite[Theorem 1]{Fan84} mentioned in the introduction (see also \cite[Appendix]{LoreauxThesis} for corrections), $0\in \mathcal D_{\rm const}(T)$. However, $0\notin W_2(T)$.
Indeed, suppose on the contrary that $0\in W_2(T)$. So there exists a two-dimensional subspace $M$ such that $P_MTP_M=0$.
In particular there exists $x=(0,x_1,x_2,\dots)\in H$ such that $\|x\|=1$ and $\langle Tx,x\rangle=0$.
However, $\langle Tx,x\rangle=\sum_{j=1}^\infty 2^{-j}|x_j|^2>0$, a contradiction.
\end{example}

While we are not able to answer this question about the convexity of $\mD_{\rm const}(\cT)$ either,
we give several statements clarifying  the structure of the set of constant diagonals of operators.
First, we describe the orbit of $\mathcal D_{\rm const}(\cT)$ under compact perturbations.

\begin{proposition}
Let $\cT=(T_1,\dots,T_n)\in B(H)^n$. Then
$$
\bigcup_{\cK\in K(H)^n} \mathcal D_{\rm const}(\cT+\cK)= W_{\rm e}(\cT).
$$
\end{proposition}

\begin{proof}
For each $\cK\in K(H)^n$ we have
$$
\mathcal D_{\rm const}(\cT+\cK)\subset W_{\rm e}(\cT+\cK)= W_{\rm e}(\cT).
$$
On the other hand, if $\la\in W_{\rm e}(\cT)$ then, by Corollary \ref{C2.2}, there exists $\cK \in K(H)^n$ (even an $n$-tuple of Schatten class operators) such that $\la\in \mD_{\rm const}(\cT+\cK)$.
\end{proof}

Similarly, we are able to locate a subset of $\mathcal D_{\rm const}(\cT)$ invariant under compact perturbations.

\begin{proposition}
Let $\cT=(T_1,\dots,T_n)\in B(H)^n$. Then
$$
\bigcap_{\cK\in K(H)^n} W(\cT+\cK)=\bigcap_{\cK\in K(H)^n} \mathcal D_{\rm const}(\cT+\cK)= \Int W_{\rm e}(\cT).
$$
\end{proposition}

\begin{proof}
Taking into account \eqref{inclusion_e}, for each $\cK\in K(H)^n$,
$$
\Int W_{\rm e}(\cT)=\Int W_{\rm e}(\cT+\cK)\subset \mathcal D_{\rm const}(\cT+\cK)
\subset W(\cT+\cK).
$$
Conversely, let $\la\in W_{\rm e}(\cT)\setminus \Int W_{\rm e}(\cT)$. Without loss of generality, by using \eqref{compactp} and considering a suitable compact perturbation, we may assume that $W_{\rm e}(\cT)=\overline{W(\cT)}$. We may also assume, by a suitable translation, that $\la=0$. Now let $\cT=(T_1,\dots,T_n)\in B(H)^n$, $0\in W_{\rm e}(\cT)=\overline{W(\cT)}$, $0\notin\Int W_{\rm e}(\cT)$.
The numerical range of the $n$-tuple $\cT$ can be identified with the numerical range of the $(2n)$-tuple $(\Re T_1,\Im T_1,\dots,\Re T_n,\Im T_n)$ of selfadjoint operators.
By a suitable rotation in $\RR^{2n}\sim\CC^n$ we may assume that $\langle (\Re T_1)x,x\rangle\ge 0$ for all $x\in H$ and $0\in \overline{W(\cT)}$.
Let $K_1=\diag(1,2^{-1},2^{-2},\dots)$ (in any orthonormal basis). Clearly $K_1$ is a compact operator on $H$
and $\langle (\Re T_1 +K_1)x,x\rangle >0$ for all $x\ne 0$. So $0\notin W(\Re T_1+K_1)$ and then $0\notin \mathcal D_{\rm const}(\cT+\cK)$
for the tuple $\cK=(K_1, 0, \dots, 0).$
Thus,
$$
\bigcap_{\cK\in K(H)^n} \mathcal D_{\rm const}(\cT+\cK)\subset
\bigcap_{\cK\in K(H)^n} W(\cT+\cK)\subset
\Int W_{\rm e}(\cT).
$$
\end{proof}

Note that in several specific cases $\mathcal D_{\rm const}(T)$ allows an explicit description.
For instance, by \cite{Bourin03}, if $T \in B(H)$ then $W(T)$ is relatively open (i.e., it is a single
point, an open segment or an open set in $\mathbb C$) if and only if $W(T)=\mathcal D_{\rm const}(T).$
The class of operators with open $W(T)$ is substantial. For instance,
it includes weighted periodic shifts and a number of Toeplitz operators.

\section{Operator-valued diagonals: Blaschke-type condition}

Let $T\in B(H)$ satisfy $W_{\rm e}(T)\supset \overline {\DD}$. Let $(C_k)_{k=1}^{\infty}$ be a sequence of contractions on Hilbert spaces $H_k$
(possibly different from $H$) such that  $\sup_k\|C_k\|<1$.
By \cite[Theorem 2.1]{Bourin03}, the operator $T$ has a "pinching" $\bigoplus_{k=1}^\infty C_k$, i.e., there exist mutually orthogonal subspaces $K_k\subset H$ such that $H=\bigoplus_{k=1}^\infty K_k$ and $P_{K_k}T \vert_{K_k}\usim C_k$ for all $k$.

In Theorem \ref{T3.1}, which can be considered as an operator-valued version of Theorem \ref{T1.1},
 we show that the assumption $\sup_{k \ge 1} \|C_k\|<1$ can be replaced by a much weaker Blaschke-type condition $\sum_{k=1}^\infty (1-\|C_k\|)=\infty$
resembling a similar condition \eqref{blaschke} above. Clearly, the operator-valued version of Theorem \ref{T1.1} is more involved
and its proof requires new arguments. However, the scheme of the proof is similar
to the one used in Section \ref{diagon}.

First, we will need an auxiliary lemma.
\begin{lemma}\label{orthogonal}
Let $T\in B(H)$ with $W_{\rm e}(T)\supset \overline {\DD}$.
Then there exist mutually orthogonal subspaces $H_k\subset H, k\in\NN,$ such that
$$
W_{\rm e}\bigl(P_{H_k}T\vert_{H_k}\bigr)\supset\overline{\DD}, \qquad k\in\NN.
$$
\end{lemma}

\begin{proof}
Let $f:\NN\to\NN^3$ be a bijection, and write  $f(k)=(f_1(k),f_2(k),f_3(k))$. Let $(w_s)_{s=1}^{\infty}\subset\TT$ be a sequence dense in $\TT$.
Inductively we construct a sequence of mutually orthogonal unit vectors $(x_k)_{k=1}^{\infty}\in H$ such that
$$
\bigl|\langle Tx_k,x_k\rangle-w_{f_1(k)}\bigr|<\frac{1}{f_2(k)}.
$$
Now suppose that  $(x_k)_{k=1}^{\infty}\subset H$ is constructed in this way. For $m\in\NN$ define
$$
H_m=\bigvee\{x_k:f_3(k)=m\},
$$
and note that the subspaces $H_m, m\in\NN,$ are mutually orthogonal. Let $s,m\in\NN$ be fixed.
Then $\{u_j:=x_{f^{-1}(s,j,m)}: j \ge 1\}$ form an orthonormal sequence in $H_m$ such that
$$
\bigl|\langle Tu_j,u_j\rangle-w_s\bigr|<\frac{1}{j}
$$
for all $j\in\NN$. Thus $w_s\in W_{\rm e}(P_{H_m}T\vert_{H_m})$ and, since this holds for all $s, m \in \NN,$ we have $W_{\rm e}(P_{H_m}T\vert_{H_m})\supset\overline{\DD}$ for all $m\in\NN$.
\end{proof}

So everything is prepared for the proof of Theorem \ref{T3.1} and we give it below.
\bigskip

\emph{Proof of Theorem \ref{T3.1}}\,\,\,
Let $H_m, m\in\NN,$ be the subspaces given by Lemma \ref{orthogonal},
and let $(y_m)_{m=1}^\infty$ be a sequence of unit vectors in $H$ such that $\bigvee_my_m=H$. Using the assumption, we can find mutually disjoint sets $A_m, m\in\NN,$ such that $\bigcup_{m=1}^\infty A_m=\NN$ and
$$
\sum_{k\in A_m}(1-\|C_k\|)=\infty
$$
for all $m\in\NN$.
Next we construct mutually orthogonal subspaces $K_k$, $k\in\NN,$ such that
$\dim K_k/(H_k\cap K_k)\le 1$, $P_{K_k}T\vert_{K_k}\usim C_k$ and
\begin{equation}\label{(1)}
\ln\dist^2\Bigl\{y_m,\bigvee_{k\le N}K_k\Bigr\}\le-\sum_{k\le N, k\in A_m}\frac{1-\|C_k\|}{16\|T\|}
\end{equation}
for all $m, N \in\NN$.
Note that as a consequence $\bigoplus_{k=1}^\infty K_k=H$. Indeed, for each $m\in\NN$ we have
$$
\lim_{N\to\infty}\ln\dist^2\Bigl\{y_m,\bigvee_{k\le N}K_k\Bigr\}\le
-\lim_{N\to\infty}\sum_{k\le N, k\in A_m}\frac{1-\|C_k\|}{16\|T\|}=-\infty.
$$
So $\dist\Bigl\{y_m,\bigvee_{k\in\NN}K_k\Bigr\}=0$ and $y_m\in\bigvee_{k\in\NN}K_k$. Hence
$H=\bigoplus_{k=1}^\infty K_k$.

The sequence of  mutually orthogonal subspaces $(K_k)_{k=1}^{\infty}$
satisfying \eqref{(1)} will be constructed inductively.
Set formally $K_0=\{0\},$ let $N\in\NN$ and suppose that the subspaces $K_k, k\le N-1,$ have already been constructed. Then $N\in A_m$
for some $m \in \mathbb N.$
Write $y_m=a+tb$ where
$$a\in\bigvee_{k\le N-1}K_k, \quad b\perp \bigvee_{k\le N-1}K_k, \quad \|b\|=1$$
and $t=\dist\Bigl\{y_m,\bigvee_{k\le N-1}K_k\Bigr\}\le 1.$

Setting $$H_N'=H_N\cap\{b,Tb,T^*b\}^\perp\cap\bigcap_{k=1}^{N-1}K_k^\perp,$$ observe that $\overline{\DD} \subset W_{\rm e}(P_{H_N'}T\vert_{H_N'})$
since $\overline{\DD} \subset W_{\rm e}(P_{H_N}T\vert_{H_N})$ and $H_N'$ is a subspace of finite codimension in $H_N$.

Now let $L_N$ be given by the theorem assumptions. Fix any unit vector $x\in L_N,$ and let $P\in B(L_N)$ be the orthogonal projection onto the one-dimensional subspace generated by $x$. If
$$
\rho=:\sqrt{\frac{1-\|C_N\|}{16\|T\|}}.
$$
then $\rho\le\frac{1}{4}$ and $\sqrt{1-\rho^2}\ge \frac{1}{2}$. Consider
\begin{align*}
C_N'=&(I-P)C_N(I-P)+\frac{1}{\sqrt{1-\rho^2}}(I-P)C_NP\\
+&\frac{1}{\sqrt{1-\rho^2}}PC_N(I-P)
 +\frac{PC_NP-\rho^2\langle Tb,b\rangle P}{1-\rho^2}\\
=&C_N+\Bigl(\frac{1}{\sqrt{1-\rho^2}}-1\Bigr)(I-P)C_NP+\Bigl(\frac{1}{\sqrt{1-\rho^2}}-1\Bigr)PC_N(I-P)\\
+&\Bigl(\frac{1}{1-\rho^2}-1\Bigr)PC_NP-\frac{\rho^2}{1-\rho^2}\langle Tb,b\rangle P,
\end{align*}
where $I$ denotes the identity operator on $L_N$.
We have
\begin{align*}
\|C_N'\|\le&
\|C_N\|+2\Bigl(\frac{1}{\sqrt{1-\rho^2}}-1\Bigr)+\Bigl(\frac{1}{1-\rho^2}-1\Bigr)+\frac{\rho^2}{1-\rho^2}\|T\|\\
\le&
\|C_N\|+2\frac{1-\sqrt{1-\rho^2}}{\sqrt{1-\rho^2}}+\frac{\rho^2}{1-\rho^2}+\frac{\rho^2}{1-\rho^2}\|T\|\\
\le&
\|C_N\|+4\bigl(1-\sqrt{1-\rho^2}\bigr)+\frac{2\|T\|\rho^2}{1-\rho^2}\\
\le&
\|C_N\|+4\rho^2+4\rho^2\|T\|\\
\le&
\|C_N\|+8\rho^2\|T\|\\
\le&
\|C_N\|+\frac{1-\|C_N\|}{2}<1.
\end{align*}
So, by for example Bourin's ``pinching'' theorem \cite[Theorem 2.1]{Bourin03} (mentioned in the introduction), there exists a subspace $K_N'\subset H_N'$ such that $P_{K_N'}T\vert_{K_N'}\usim C_N'$.
Thus there exists a unitary operator $U:L_N\to K_N'$ such that
$$
P_{K_N'}T\vert_{K_N} = U C_N'U^{-1}.
$$
Set
$$
v=\sqrt{1-\rho^2} Ux+\rho b.
$$
Since $Ux\in K_N', K_N' \subset H_N'$ and $H_N'\perp b$, we have $\|v\|=1$ and $v\perp U(I-P)L_N$.
Let $$K_N:=U(I-P)L_N\vee \{v\},$$ and
let $V:L_N\to K_N$ be defined by $$V\vert_{(I-P)L_N}:=U\vert_{(I-P)L_N} \qquad \text{and} \qquad Vx:=v.$$
By construction $\dim K_N/(K_N\cap H_N)\le 1$.
Clearly $V$ is a unitary operator, and $K_N\perp\bigvee_{k\le N-1}K_k$. For any $z=u+sv$, where $u\in (I-P)L_N$ and $s\in\CC,$ we have
\begin{align*}
\bigl\langle V^{-1}(P_{K_N}&T\vert_{K_N}) Vz,z\bigr\rangle\\
=&
\langle T Vz,Vz\rangle=
\bigl\langle T(Uu+sv),Uu+sv\bigr\rangle\\
=&\bigl\langle T(Uu+s\sqrt{1-\rho^2} Ux+s\rho b), Uu+s\sqrt{1-\rho^2} Ux+s\rho b\bigr\rangle\\
=&\bigl\langle T(Uu+s\sqrt{1-\rho^2} Ux), Uu+s\sqrt{1-\rho^2} Ux\bigr\rangle
+s^2\rho^2\langle Tb,b\rangle\\
=&
\bigl\langle C_N'(u+s\sqrt{1-\rho^2}x),u+s\sqrt{1-\rho^2}x\bigr\rangle+s^2\rho^2\langle Tb,b\rangle\\
=&\bigl\langle (I-P)C_N'(I-P)z,z\bigr\rangle+\sqrt{1-\rho^2}\bigl\langle (I-P)C_N'Pz,z\bigr\rangle\\
 +&\sqrt{1-\rho^2}\bigl\langle PC_N'(I-P)z,z\bigr\rangle
+(1-\rho^2)\langle PC_N'Pz,z\rangle+\rho^2\langle Tb,b\rangle\langle Pz,z\rangle\\
=&
\langle C_Nz,z\rangle.
\end{align*}
Hence $P_{K_N}T\vert_{K_N}\usim C_N$.

Moreover, since  $\langle b,v\rangle=\rho$ and
$$
\dist^2\Bigl\{y_m,\bigvee_{k\le N}K_k\Bigr\}=t^2\dist^2\Bigl\{b,\bigvee v\Bigr\}=t^2(1-\rho^2)=
t^2\Bigl(1-\frac{1-\|C_N\|}{16 \|T\|}\Bigr),
$$
we infer that
$$
\ln\frac{\dist^2\Bigl\{y_m,\bigvee_{k\le N}K_k\Bigr\}}{\dist^2\Bigl\{y_m,\bigvee_{k\le N-1}K_k\Bigr\}}
=\ln\Bigl(1-\frac{1-\|C_N\|}{16 \|T\|}\Bigr)\le
-\frac{1-\|C_N\|}{16\|T\|}
$$
and thus
$$
\ln \dist^2\Bigl\{y_m,\bigvee_{k\le N}K_k\Bigr\}\le
-\sum_{k\le N, k\in A_m}\frac{1-\|C_k\|}{16\|T\|}.
$$
This finishes the proof. \hfill $\Box$

\bigskip

Replacing the numerical range condition $W_{\rm e}(T)\supset\overline{\DD}$ in Theorem \ref{T3.1}
by the spectral assumption $\widehat\si(T)\supset\overline{\DD},$ we can put
 Theorem \ref{T3.1} in a more demanding context of tuples of powers of $T.$
 To this aim we will need the next lemma.

\begin{lemma}\label{L3.2}
Let $T\in B(H)$ satisfy $\widehat\si(T)\supset\overline{\DD}$. Let $n\in\NN$, let $L$ be a separable Hilbert space, and let $C,A_1,\dots,A_n\in B(L)$ be such that $\|C\|<1$ and $\|A_j\|\le\frac{(1-\|C\|)^n}{n 2^{2n+4}}, j=1,\dots,n$. Then there exists a subspace $K\subset H$ such that
$$
P_K(T,T^2,\dots,T^n)\vert_K\usim (C+A_1, C^2+A_2,\dots,C^n+A_n).
$$
\end{lemma}

\begin{proof}
Set
$$c:=\|C\|, \quad d:=\frac{(1-c)^n}{n 2^{2n+4}}, \quad c':=c+\frac{c(1-c)^n}{2^{n+1}}, \quad \text{and} \quad d':=\frac{1}{2^{n+1}}.$$

Note that
\begin{equation}\label{(2)}
\frac{c}{c'}+2n\frac{d}{d'}<1.
\end{equation}
Indeed, since $\frac{1}{1+a}<1-\frac{a}{2}$ for all $a\in [0,1),$ we have
$$
\frac{1}{1+\frac{(1-c)^n}{2^{n+1}}} +\frac{(1-c)^n}{2^{n+2}}<1,
$$
which is equivalent to \eqref{(2)}.

Choose
\begin{equation}\label{combi1}
\e':=\frac{1}{4n}\Bigl(1-\frac{c}{c'}-2n\frac{d}{d'}\Bigr)>0 \quad \text{and} \quad  \e \in \Bigl(0, \frac{\e'}{n2^n}\Bigr).
\end{equation}
Let $U$ be the minimal unitary (power) dilation of the contraction $\frac{C}{c}$
on the separable Hilbert space $M\supset L.$
Extend the operators $A_1,\dots,A_n$ to $\tilde A_1,\dots,\tilde A_n\in B(M)$ by defining
$\tilde A_j\vert_L=A_j\vert_L$ and $\tilde A_j\vert_{M\ominus L}=0$.
We show that $\bigl(cU+\tilde A_1,\dots,(cU)^n+\tilde A_n\bigr)$ is a convex combination of $n$-tuples belonging to the set $\cM(W_\infty(\cT))$
(defined in Section \ref{prelim}).

By  Weyl-von Neumann's theorem, $cU$ can be written as $cU=D+K,$
where $D\in B(M)$ is a diagonal operator with
$\sigma(D)=\sigma_{\rm e}(D)=\sigma_{\rm e}(cU)$, so the entries of $D$ are of moduli
$c$,
and $K\in B(M)$ is a compact operator, $\|K\|\le\e$.
(See e.g. \cite[Corollary 38.4]{Conway}, where the proof is given for a selfadjoint version of the theorem.) For $j=1,\dots,n$ let $K_j=(cU)^j-D^j$.
We have
$$
K_j=\sum_{s=0}^{j-1} (cU)^s K D^{j-s-1}, \qquad j \ge 1.
$$
So $K_j$ is a compact operator for each $j=1,\dots,n$ and $\|K_j\|\le n\e$.

Similarly, we have $\sup_{1\le j\le n}\|A_j\|\le d,$ so again by Weyl-von Neumann's theorem (\cite[Corollary 38.4]{Conway}), for every $j$ the operators $\Re\tilde A_j$ and $\Im \tilde A_j$ can be written as $\Re\tilde A_j=D_j'+K_j'$ and $\Im\tilde A_j=D_j''+K_j''$,
where $D_j'$ and $D_j''$ are diagonal operators with entries of moduli not exceeding $d,$ and $K_j'$ and $K_j''$ are compact selfadjoint operators with $\|K_j'\|\le\e$ and $\|K_j''\|\le\e$.
Thus we can write
\begin{align}\label{combi2}
\bigl(cU+\tilde A_1,&\dots,(cU)^n+\tilde A_n\bigr)\\
=&
\frac{c}{c'}\Bigl(\frac{c'}{c}\bigl(D,\dots,D^n\bigr)\Bigr)\notag \\
+&\sum_{j=1}^n\frac{d}{d'}\Bigl(\frac{d'}{d}\bigl(\underbrace
{0,\dots,0}_{j-1},D_j'+iD_j'',0,\dots,0\bigr)\Bigr)\notag \\
+&\sum_{j=1}^n\e'
\Bigl(\frac{1}{\e'}\bigl(\underbrace{0,\dots,0}_{j-1},\Re K_j+i\Im K_j+K_j'+iK_j'',0,\dots,0\bigr)\Bigr).\notag
\end{align}
Write for short $\cT=(T,\dots,T^n)$.
By construction,  $D= {\rm diag}[(d_k)_{k=1}^\infty]$ with $|d_k|=c$ for all $k.$
Note that if  $z\in\CC$, $|z|=c,$ then
$$
\Bigl\|\frac{c'}{c}(z,z^2,\dots,z^n)-(z,z^2,\dots,z^n)\Bigr\|=
\Bigl(\frac{c'}{c}-1\Bigr)\cdot c=c'-c=
\frac{(1-c)^nc}{2^{n+1}}.
$$
Hence, since $\mathbb D \subset \widehat \sigma(T),$  we infer by Corollary \ref{C4.10} that
$$\frac{c'}{c}\bigl(D,\dots,D^n\bigr)=\diag \Bigl [\frac{c'}{c}(d_k, \dots, d_k^n)_{k=1}^{\infty} \Bigr]$$
is a jointly diagonal $n$-tuple with all of its entries $\frac{c'}{c} (d_k, \dots, d_k^n)$ in $W_{\infty}(\cT).$
So
$\frac{c'}{c}(D,\dots,D^n\bigr)\in\cM(W_\infty(\cT))$ by the definition of $\cM(W_\infty(\cT)).$

Similarly,
$$
\frac{d'}{d}\bigl(\underbrace
{0,\dots,0}_{j-1},D_j',0,\dots,0\bigr)
$$
is a diagonal tuple. The moduli of diagonal entries of $\frac{d'}{d} D_j'$ do not exceed $d'=\frac{1}{2^{n+1}}.$
Hence, by the spectral assumption $\overline{\mathbb D} \subset \widehat \sigma(T)$ and Corollary \ref{C4.10},
$$
\frac{d'}{d}\bigl(\underbrace
{0,\dots,0}_{j-1},D_j',0,\dots,0\bigr)\in\cM(W_\infty(\cT)).
$$
In the same way
$$
\frac{d'}{d}\bigl(\underbrace
{0,\dots,0}_{j-1},D_j'',0,\dots,0\bigr)\in\cM(W_\infty(\cT)).
$$
Finally, $\Re K_j, \Im K_j, K_j',K_j'', 1\le j \le n,$ are compact selfadjoint, and thus unitarily diagonalisable operators.
The moduli of diagonal entries of these operators do not exceed $\e$. Thus the tuples
\begin{align*}
&\frac{1}{\e'}\bigl(\underbrace{0,\dots,0}_{j-1},\Re K_j,0,\dots,0\bigr),\qquad
\frac{1}{\e'}\bigl(\underbrace{0,\dots,0}_{j-1},i\Im K_j,0,\dots,0\bigr),\\
&\frac{1}{\e'}\bigl(\underbrace{0,\dots,0}_{j-1}, K_j',0,\dots,0\bigr)\qquad\hbox{and}\qquad
\frac{1}{\e'}\bigl(\underbrace{0,\dots,0}_{j-1},iK_j'',0,\dots,0\bigr)
\end{align*}
are diagonal (each in its own basis),  and operator norms of their entries are not larger than $\frac{\e}{\e'}<\frac{1}{2^n}$. Hence the tuples belong to $\cM(W_\infty(\cT))$.

 Thus, in view of \eqref{combi1} and \eqref{combi2}, $$\bigl(cU+\tilde A_1,\dots,(cU)^n+\tilde A_n\bigr) \in \conv\, \cM(W_\infty(\cT)).$$
Now Proposition \ref{propm} implies that there exists a subspace $K'\subset H$ such that
$$
\bigl(cU+\tilde A_1,\dots,(cU)^n+\tilde A_n\bigr)\usim P_{K'}(T,\dots,T^n)\vert_{K'},
$$
that is $V(cU+\tilde A_j)V^{-1} =  P_{K'} T^j\vert_{K'}$ for some unitary $V: M \to K'$ and all $1 \le j\le n.$
Since
$$
P_L\bigl(cU+\tilde A_1,\dots,(cU)^n+\tilde A_n\bigr)\vert_L=
(C+A_1,\dots,C^n+A_n),
$$
we let $K=V(L)$ and infer that $P_K T^j \vert_K=\tilde V^{-1} (C^j+A_j) \tilde V, 1 \le j \le n,$ where $\tilde V = V \vert_L: L \to K,$
hence the theorem follows.
\end{proof}

For a sequence of Hilbert space contractions $(C_k)_{k=1}^{\infty}$ with norms not approaching $1$ too fast,
the following statement yields pinchings $(C_k,\dots,C_k^n)$ for a tuple $(T,\dots,T^n), T \in B(H),$
if the spectrum of $T$ is sufficiently large. Its assumptions are close to be optimal even if $n=1,$ see \cite{Bourin03}.

\begin{theorem}\label{theoremspectral}
Let $T\in B(H)$ be such that $\widehat\si(T)\supset\overline{\DD}$, and let $n\in\NN$. Let $L_k, k \in \mathbb N,$ be separable Hilbert spaces, and let
$C_k\in B(L_k), k \in \mathbb N,$ satisfy $\sum_{k=1}^\infty (1-\|C_k\|)^n=\infty$.
Then there are mutually orthogonal subspaces $K_k, k \in \mathbb N,$ of $H$ such that
$$
H=\bigoplus_{k=1}^\infty K_k \qquad \text{and} \qquad P_{K_k}(T,\dots,T^n)\vert_{K_k}\usim (C_k,\dots,C_k^n)
$$
for all $k\in\NN$.
\end{theorem}

\begin{proof}
The proof is analogous to that of Theorem \ref{T3.1}, so we present it only briefly.

As in the proof of Theorem \ref{T3.1}, one can find mutually orthogonal subspaces
$H_m\subset H, m\in\NN,$ such that $$W_{\rm e}\bigl(P_{H_m}(T,\dots,T^n)\vert_{H_m}\bigr) \supset \overline{\mathbb D}.
$$

Let $(y_m)_{m=1}^\infty$ be a sequence of unit vectors in $H$ such that $\bigvee_my_m=H$. We can find a sequence of mutually disjoint sets $(A_m)_{m=1}^\infty$ such that $\bigcup_{m=1}^\infty A_m=\NN$ and
$$
\sum_{k\in A_m}(1-\|C_k\|)^n=\infty
$$
for all $m\in\NN$.
We construct mutually orthogonal subspaces $K_k\subset H, k\in\NN,$ satisfying $P_{K_k}(T,\dots,T^n)\vert_{K_k}\usim (C_k,\dots,C_k^n)$, $\dim K_k/(H_k\cap K_k)\le 1$ and
\begin{equation*}
\ln\dist^2\Bigl\{y_m,\bigvee_{k\le N}K_k\Bigr\}\le-\sum_{k\le N, k\in A_m}\frac{(1-\|C_k\|)^n}{n2^{2n+7}\|T\|^n}
\end{equation*}
for all $m, N\in\NN$.

Set formally $K_0=\{0\},$ let $N\in\NN$ and suppose that the subspaces $K_k, k\le N-1,$ have already been constructed. Let $N\in A_m$.
Let $y_m=a+tb$ where $$a\in\bigvee_{k\le N-1}K_k, \quad b\perp \bigvee_{k\le N-1}K_k, \quad \|b\|=1$$ and $t=\dist\Bigl\{y_m,\bigvee_{k\le N-1}K_k\Bigr\}\le 1.$

Define $$H_N'=H_N\cap\{b,T^jb,T^{*j}b:j=1,\dots,n\}^\perp\cap\bigcap_{k=1}^{N-1}K_k^\perp,$$ and note that $W_{\rm e}(P_{H_N'}T\vert_{H_N'})\supset\overline{\DD}$ since $H_N$ is of finite codimension in $H_N'$.

Fix any unit vector $x\in L_N$ and let $P\in B(L_N)$ be the orthogonal projection onto the one-dimensional subspace generated by $x$. If
$$
\rho=
\sqrt{\frac{(1-\|C_k\|)^n}{n2^{2n+7}\|T\|^n}}
$$
then
 $\rho\le\frac{1}{4}$ and $\sqrt{1-\rho^2}\ge \frac{1}{2}$. For each $j=1,\dots,n$ consider
 \begin{align*}
C_{N,j}'=&(I-P)C_N^j(I-P)+\frac{1}{\sqrt{1-\rho^2}}(I-P)C_N^jP\\
+&\frac{1}{\sqrt{1-\rho^2}}PC_N^j(I-P)
+\frac{PC_N^jP-\rho^2\langle T^j b,b\rangle P}{1-\rho^2},
\end{align*}
where $I$ denotes the identity operator on $L_N$.
We have
\begin{align*}
\|C_{N,j}'-C_N^j\|\le&
2\Bigl(\frac{1}{\sqrt{1-\rho^2}}-1\Bigr)+\Bigl(\frac{1}{1-\rho^2}-1\Bigr)+\frac{\rho^2}{1-\rho^2}\|T\|^j\\
\le& 8\rho^2\|T\|^j.
\end{align*}

By Lemma \ref{L3.2}, there exists a subspace $K_N'\subset H_N'$ such that
$$
P_{K_N'}(T,\dots,T^n)\vert_{K_N'}\usim (C_{N,1}',\dots,C_{N,n}'),
$$
i.e.,
there exists a unitary operator $U:L_N\to K_N'$ such that
$$
U^{-1}P_{K_N'}(T,\dots,T^n)\vert_{K_N'} U=(C_{N,1}',\dots,C_{N,n}').
$$
Set
$$
v=\sqrt{1-\rho^2} Ux+\rho b.
$$
Since $Ux\in K_N'\subset H_N'$ and $H_N'\perp b$, we have $\|v\|=1$ and $v\perp U(I-P)L_N$.
Let $$K_N=U(I-P)L_N\vee \{v\}$$ and note that $K_N\perp\bigvee_{k\le N-1}K_k$.
Define a unitary operator $V:L_N\to K_N$ by $$V\vert_{(I-P)L_N}:=U\vert_{(I-P)L_N} \qquad \text{and} \qquad Vx:=v.$$

 For all $z=u+sv$, $u\in (I-P)L_N$, $s\in\CC$, and $1\le j\le n$ we have
\begin{align*}
\bigl\langle V^{-1}(P_{K_N}&T^j\vert_{K_N}) Vz,z\bigr\rangle\\
=&
\bigl\langle T^j(Uu+sv),Uu+sv\bigr\rangle\\
=&\bigl\langle T^j(Uu+s\sqrt{1-\rho^2} Ux), Uu+s\sqrt{1-\rho^2} Ux\bigr\rangle
+s^2\rho^2\langle T^j b,b\rangle\\
=&\bigl\langle (I-P)C_{N,j}'(I-P)z,z\bigr\rangle+\sqrt{1-\rho^2}\bigl\langle (I-P)C_{N,j}'Pz,z\bigr\rangle\\ +&\sqrt{1-\rho^2}\bigl\langle PC_{N,j}'(I-P)z,z\bigr\rangle\\
+&(1-\rho^2)\langle PC_{N,j}'Pz,z\rangle+\rho^2\langle T^jb,b\rangle\langle Pz,z\rangle\\
=&
\langle C_N^jz,z\rangle.
\end{align*}
Hence $$P_{K_N}(T,\dots,T^n)\vert_{K_N}\usim (C_N,\dots,C_N^n).$$

Moreover, we have $\langle b,v\rangle=\rho$ and
$$
\dist^2\Bigl\{y_m,\bigvee_{k\le N}K_k\Bigr\}=t^2\dist^2\Bigl\{b,\bigvee v\Bigr\}=t^2(1-\rho^2).
$$
Thus, recalling the definition of $\rho$,
$$
\ln\frac{\dist^2\Bigl\{y_m,\bigvee_{k\le N}K_k\Bigr\}}{\dist^2\Bigl\{y_m,\bigvee_{k\le N-1}K_k\Bigr\}}
=\ln\Bigl(1-\frac{(1-\|C_N\|)^n}{n2^{2n+7}\|T\|^n}\Bigr)\le
-\frac{(1-\|C_N\|)^n}{n2^{2n+7}\|T\|^n}.
$$
As in Theorem \ref{T3.1}, this finishes the proof.

\end{proof}
\section{Acknowledgments}
We would like to thank the anonymous referee for helpful remarks
improving the presentation.


\begin{thebibliography}{100}

\bibitem{AndersonThesis} J. Anderson,
\emph{Derivations, commutators and the essential numerical range,} Ph.D. Thesis, Indiana University, 1971.

\bibitem{Anderson71} J.H. Anderson, J.G. Stampfli,
\emph{Commutators and compressions,} Israel J. Math. \textbf{10} (1971), 433--441.

\bibitem{Antezana07} J. Antezana, P. Massey, M. Ruiz, D. Stojanoff,
\emph{The Schur-Horn theorem for operators and frames with prescribed norms and frame operator,} Illinois J.
Math. \textbf{51} (2007), 537--560.

\bibitem{Argerami13} M. Argerami, P. Massey, \emph{Schur-Horn theorems in ${\rm II}_\infty$-factors,} Pacific J. Math. \textbf{261} (2013), 283--310.

\bibitem{Argerami15} M. Argerami, \emph{Majorisation and the Carpenter's theorem,}
Integral Equations Operator Theory \textbf{82} (2015),  33--49.

\bibitem{Arveson07} W. Arveson, \emph{Diagonals of normal operators with finite spectrum,} Proc. Natl. Acad. Sci. USA \textbf{104} (2007), 1152--1158.

\bibitem{Arveson06} W. Arveson, R.V. Kadison, \emph{Diagonals of self-adjoint operators,} in:
 Operator Theory, Operator Algebras, and Applications, Contemp. Math., vol. \textbf{414}, AMS, Providence, RI, 2006, pp. 247--263.

 \bibitem{BercoviciFoias} H. Bercovici, C Foias,  A. Tannenbaum, \emph{The structured singular value for linear input/output operators,} SIAM J. Control Optim. \textbf{34} (1996),  1392--1404.

\bibitem{Bervovici83}   H. Bercovici, C. Foias, and C. Pearcy, \emph{Dilation theory and systems of simultaneous equations in the predual of an operator algebra, I} Michigan Math. J. \textbf{30} (1983),  335--354.

\bibitem{Bhat14} B. V. R. Bhat, M. Ravichandran, \emph{The Schur-Horn theorem for operators with finite spectrum,}
Proc. Amer. Math. Soc. \textbf{142} (2014), 3441--3453.


\bibitem{Bourin03} J.C. Bourin, \emph{Compressions and pinchings,} J. Operator Theory \textbf{50} (2003), 211--220.

\bibitem{Bourin16}  J.-C. Bourin and E.-Y. Lee, \emph{Pinchings and positive linear maps,} J. Funct. Anal. \textbf{270} (2016),  359--374.

\bibitem{Bownik11} M. Bownik, J. Jasper, \emph{Characterization of sequences of frame norms,}
J. Reine Angew. Math. \textbf{654} (2011), 219--244.

\bibitem{Bownik14} M. Bownik, J. Jasper, \emph{Constructive proof of the Carpenter's theorem,} Canad. Math. Bull. \textbf{57} (2014), 463--476.

\bibitem{Bownik15a} M. Bownik, J. Jasper, \emph{Existence of frames with prescribed norms and frame operator,} Excursions in harmonic analysis. Vol. \textbf{4}, 103--117, Appl. Numer. Harmon. Anal., Birkh\"auser/Springer, Cham, 2015.

\bibitem{Bownik15} M. Bownik, J. Jasper, \emph{The Schur-Horn theorem for operators with finite spectrum,} Trans. Amer. Math. Soc.,  \textbf{367} (2015),  5099--5140.

\bibitem{Conway}  J. B. Conway, \emph{A course in operator theory,} GTM \textbf{21}, AMS, Providence, RI, 2000.

\bibitem{Dragan} C. Dragan, V. Kaftal, \emph{Sums of equivalent sequences of positive operators in von Neumann factors,}
Houston J. Math. \textbf{42} (2016), 991--1017.

\bibitem{Fan84} P. Fan, \emph{On the diagonal of an operator,}
Trans. Amer. Math. Soc. \textbf{283} (1984), 239--251.

\bibitem{Fan94} P. Fan, C.-K. Fong, \emph{An intrinsic characterization for zero-diagonal operators,}
Proc. Amer. Math. Soc. \textbf{121} (1994), 803--805.

\bibitem{Fan87} P. Fan, C.-K. Fong, D. Herrero, \emph{On zero-diagonal operators and traces,}
Proc. Amer. Math. Soc. \textbf{99} (1987), 445–-451.

\bibitem{Fillmore72} P. A. Fillmore, J. G. Stampfli, J.P. Williams, \emph{On the essential numerical range, the essential spectrum, and a problem of Halmos,} Acta Sci. Math. (Szeged) \textbf{33} (1972), 179--192.

\bibitem{Fong86} C.-K. Fong, \emph{Diagonals of nilpotent operators,} Proc. Edinburgh Math. Soc. \textbf{29} (1986), 221--224.

\bibitem{Fong87} C.-K. Fong, H. Radjavi, P. Rosenthal, \emph{Norms for matrices and operators,}
J. Operator Theory \textbf{18} (1987),  99--113.

\bibitem{Herrero91} D. Herrero, \emph{The diagonal entries of a Hilbert space operator,} Rocky Mountain J. Math. \textbf{21} (1991), 857--864.

\bibitem{Jasper13} J. Jasper, \emph{The Schur-Horn theorem for operators with three point spectrum,} J. Funct. Anal. \textbf{265} (2013),
 1494--1521.

\bibitem{Loreaux18} J. Jasper, J. Loreaux,  G. Weiss, \emph{Thompson's theorem for compact operators and diagonals of unitary operators,} Indiana Univ. Math. J., \textbf{67} (2018), 1--27.

\bibitem{Kadison02a} R. V. Kadison, \emph{The Pythagorean Theorem I: the finite case,} Proc. Natl. Acad. Sci. USA
\textbf{99} (2002), 4178--4184.

\bibitem{Kadison02b} R. V. Kadison, \emph{The Pythagorean Theorem II: the infinite discrete case,} Proc. Natl. Acad. Sci. USA
\textbf{99} (2002), 5217--5222.

\bibitem{Kaftal10} V. Kaftal, G. Weiss, \emph{An infinite dimensional Schur-Horn theorem and majorization theory,} J. Funct. Anal.
 \textbf{259} (2010), 3115--3162.

 \bibitem{Kaftal17} V. Kaftal, J.  Loreaux, \emph{Kadison's Pythagorean theorem and essential codimension,} Integral Equations Operator Theory \textbf{87} (2017),  565--580.

\bibitem{Kennedy17} M. Kennedy, P. Skoufranis, \emph{Thompson's theorem for $\rm II_1$ factors,}
Trans. Amer. Math. Soc. \textbf{369} (2017), 1495--1511.

\bibitem{Kennedy15} M. Kennedy, P. Skoufranis,\emph{The Schur-Horn problem for normal operators,} Proc. Lond. Math. Soc.
\textbf{111} (2015),  354--380.


\bibitem{Li09}   C.-K. Li, Y.-T. Poon, \emph{The joint essential numerical range of operators: convexity and related results,}
 Studia Math.  \textbf{194} (2009),  91--104.

\bibitem{Li-Rocky} C.-K.  Li, Y.-T. Poon, N.-S. Sze, \emph{
Elliptical range theorems for generalized numerical ranges of quadratic operators,}
Rocky Mountain J. Math. \textbf{41} (2011),  813--832.

\bibitem{Li11}  C.-K. Li, Y.-T. Poon, \emph{Generalized numerical ranges and quantum error correction,}
 J. Operator Theory \textbf{66} (2011),  335--351.

\bibitem{Loreaux15} J. Loreaux, G. Weiss, \emph{Majorization and a Schur-Horn theorem for positive compact operators, the nonzero kernel case,}
 J. Funct. Anal. \textbf{268} (2015), 703--731.

\bibitem{Loreaux16} J Loreaux, G. Weiss, \emph{Diagonality and idempotents with applications to problems in operator theory and frame theory,}
 J. Operator Theory \textbf{75} (2016), 91--118.


\bibitem{Loreaux19} J. Loreaux, \emph{Restricted diagonalization of finite spectrum normal operators and a theorem of Arveson,} arXiv:1712.06554.

\bibitem{LoreauxThesis} J. Loreaux, \emph{Diagonals of operators: majorization, a Schur-Horn theorem and zero-diagonal idempotents,} Ph. D. thesis, Ohio State University,
 2016, https://etd.ohiolink.edu.

\bibitem{Massey16} P. Massey,  M. Ravichandran,  \emph{Multivariable Schur-Horn theorems,} Proc. Lond. Math. Soc.  \textbf{112} (2016), 206--234.

\bibitem{Kadison-Singer} E. Matheron, \emph{
Le probleme de Kadison-Singer,}
Ann. Math. Blaise Pascal \textbf{22} (2015),  151--265.

 \bibitem{MullerSt}  V. M\"uller, \emph{The joint essential numerical range, compact perturbations, and the Olsen problem,}
 Studia Math. \textbf{197} (2010),  275--290.

\bibitem{Muller18} V. M\"uller, Yu. Tomilov, \emph{Circles in the spectrum and the geometry of orbits: a numerical ranges approach,}
 J. Funct. Anal. \textbf{274} (2018), 433--460.

 \bibitem{Muller19} V. M\"uller, Yu. Tomilov, \emph{Joint numerical ranges and compressions of powers of operators,}  J. London Math. Soc., to appear, DOI 10.1112/jlms.12165.

\bibitem{Muller07} V. M\"uller, \emph{Spectral Theory of Linear Operators and Spectral Systems in Banach Algebras,} second ed., Oper. Theory Adv. Appl., vol. \textbf{139}, Birkh\"auser, Basel, 2007.

\bibitem{Neumann99} A. Neumann, \emph{An infinite dimensional version of the Schur-Horn convexity theorem,} J. Funct. Anal. \textbf{161} (1999), 418--451.

\bibitem{Stout81} Q. Stout, \emph{Schur products of operators and the essential numerical range,}
 Trans. Amer. Math. Soc. \textbf{264} (1981),  39--47.

\bibitem{Tam} T.-Y. Tam, \emph{A Lie-theoretic approach to Thompson's theorems on singular values-diagonal elements and some related results,} J. London Math. Soc. \textbf{60} (1999),  431--448.

\bibitem{Tse-Rocky} S.-H. Tso,  P. Y. Wu, \emph{Matricial ranges of quadratic operators,} Rocky Mountain J. Math. \textbf{29} (1999), 1139--1152.

\bibitem{Weiss14} G. Weiss, \emph{A brief survey on: 1. infinite-dimensional Schur-Horn theorems and infinite majorization theory with applications to operator ideals; 2. $B(H)$-subideals of operators,} Algebraic methods in functional analysis,  Oper. Theory Adv. Appl., \textbf{233}, Birkh\"auser/Springer, Basel, 2014, 281--294.
\end{thebibliography}
\end{document}